\documentclass{MathPannX}

\usepackage[numbers]{natbib}
\usepackage[colorlinks]{hyperref}
\hypersetup{
  colorlinks,
  citecolor=coolblack,
  linkcolor=coolblack,
  urlcolor=coolblack}

\newcommand{\articletype}{Original Research Paper}
\newcommand{\received}{Month Day, Year}
\newcommand{\accepted}{Month Day, Year}
\newcommand{\DOI}{10.1556/2062.yyyy.-----}
\newcommand{\biblinfo}{xx /NS xx-26/ (yyyy) i, bp-ep}
\newcommand{\HandlingEditor}{Handling Editor}

\usepackage{lipsum}

\usepackage[T1]{fontenc}
\usepackage[utf8]{inputenc}
\usepackage{babel}
\usepackage{graphicx}
\usepackage{amsmath}
\usepackage{amsthm}
\usepackage{amssymb}
\usepackage{amsfonts}
\usepackage{latexsym}

\newcommand{\PP}{{\mathcal P}}

\newcommand{\ze}{\zeta}
\newcommand{\zs}{\zeta(s)}

\newcommand{\GG}{{\mathcal G}}

\newcommand{\si}{\sigma}
\newcommand{\ff}{\varphi}

\newcommand{\G}{{\mathcal G}}
\newcommand{\A}{{\mathcal A}}

\newcommand{\Cm}{{\mathcal C}_{-}}
\newcommand{\II}{{\mathcal I}}

\newcommand{\Z}{{\mathcal Z}}

\newcommand{\RR}{{\mathbb R}}
\newcommand{\CC}{{\mathbb C}}

\newcommand{\NN}{{\mathbb N}}

\newcommand{\N}{{\mathcal N}}
\newcommand{\R}{{\mathcal R}}

\newcommand{\al}{{\alpha}}
\newcommand{\ve}{{\varepsilon}}

\newcommand{\diam}{{\rm diam\,}}

\newcommand{\de}{\delta}

\reversemarginpar 

\begin{document}

\title{A Riemann-von Mangoldt-type formula for the \\ distribution of Beurling primes}

\author{\large{Szil\' ard Gy. R\' ev\' esz}$^1$}

\affiliation[1]{Alfréd Rényi Institute of Mathematics, 13 Reáltanoda street Budapest, Hungary 1053}

\newcommand{\CorrespondingAuthorEmail}{revesz@renyi.hu}

\newcommand{\KeyWords}{Beurling zeta function, analytic
continuation, arithmetical semigroups, Beurling prime number
formula, zero of the Beurling zeta function, Riemann-von Mangoldt formula}

\newcommand{\PrimarySubjectClass}{11M41}
\newcommand{\SecondarySubjectClass}{11F66, 11M36, 30B50, 30C15}

\begin{abstract}
In this paper we work out a Riemann-von Mangoldt type formula for the summatory function $\psi(x):=\sum_{g\in\GG, |g|\le x} \Lambda_{\GG}(g)$, where $\GG$ is an arithmetical semigroup (a Beurling generalized system of integers) and $\Lambda_{\GG}$ is the corresponding von Mangoldt function attaining $\log|p|$  for $g=p^k$ with a prime element $p\in \GG$ and zero otherwise.
On the way towards this formula, we prove explicit estimates on the Beurling zeta function $\zeta_{\GG}$, belonging to $\GG$, to the number of zeroes of $\zeta_\GG$ in various regions, in particular within the critical strip where the analytic continuation exists, and to the magnitude of the logarithmic derivative of $\zeta_\GG$, under the sole additional assumption that Knopfmacher's Axiom A is satisfied. We also construct a technically useful broken line contour to which the technic of integral transformation can be well applied. The whole work serves as a first step towards a further study of the distribution of zeros of the Beurling zeta function, providing appropriate zero density and zero clustering estimates, to be presented in the continuation of this paper.
\end{abstract}

\maketitle



\section{Introduction}

Beurling's theory fits well to the study of several mathematical
structures. A vast field of applications of Beurling's theory is
nowadays called \emph{arithmetical semigroups}, which are
described in detail e.g. by Knopfmacher, \cite{Knopf}.

Here $\G$ is a unitary, commutative semigroup, with a countable set
of indecomposable generators, called the \emph{primes} of $\G$ and
denoted usually as $p\in\PP$, (with $\PP\subset \G$ the set of all
primes within $\G$), which freely generate the whole of $\G$: i.e.,
any element $g\in \G$ can be (essentially, i.e. up to order of
terms) uniquely written in the form $g=p_1^{k_1}\cdot \dots \cdot
p_m^{k_m}$: two (essentially) different such expressions are
necessarily different as elements of $\G$, while each element has
its (essentially) own unique prime decomposition.

Moreover, there is a \emph{norm} $|\cdot|~: \G\to \RR_{+}$ so that
the following hold. Firstly, the image of $\G$, $|\G|\subset
\RR_{+}$ is discrete, i.e. any finite interval of $\RR_{+}$ can contain
the norm of only a finite number of elements of $\G$; thus the
function
\begin{equation}\label{Ndef}
{\N}(x):=\# \{g\in \G~:~ |g| \leq x\}
\end{equation}
exists as a finite, nondecreasing, nonnegative integer valued
function on $\RR_{+}$.

Second, the norm is multiplicative, i.e. $|g\cdot h| = |g| \cdot
|h|$; it follows that for the unit element $e$ of $\G$ $|e|=1$, and
that all other elements $g \in \G$ have norms strictly larger than 1
(otherwise the different elements $g^m$ -- which are really
different by their different unique prime decomposition -- would
have a sequence of norms identically $1$ or converging to $0$).

There are many structures fitting into this general theory, e.g.
the category of all finite Abelian groups, the ideals ${\mathcal
I}_K$ of an algebraic number field $K$ over the rational numbers,
algebraic structures like e.g. semisimple rings, graphs of some
required properties, finite pseudometrizable topological spaces,
symmetric Riemannian manifolds, compact Lie groups etc., see
Knopfmacher's book, pages 11-22.

In this work we assume the so-called \emph{"Axiom A"} of
Knopfmacher see pages 73-79 of his fundamental book \cite{Knopf},
in its normalized form.

\begin{definition} It is said that ${\N}$ (or, loosely speaking, $\ze$)
satisfies \emph{Axiom A} -- more precisely, Axiom
$A(\kappa,\theta)$ with the suitable constants $\kappa>0$ and
$0<\theta<1$ -- if we have\footnote{The usual formulation uses the more natural version $\R(x):= \N(x)-\kappa x$. However, our version is more convenient with respect to the initial values at 1, as we here have $\R(1-0)=0$. All respective integrals of the form $\int_0$ will be understood as integrals from $1-0$, and thus we can avoid considering endpoint values in the partial integration formulae. Alternatively, we could have taken also $\N(x)$ left continuous, and integrals from 1 in the usual sense: also with this convention we would have $\R(1)=0$.} for the remainder term
\begin{align}\label{Athetacondi}\notag & \R(x):= \N(x)-\kappa (x-1)
\\
\textrm{the estimate} \notag
\\ & \left| \R(x) \right|  \leq A x^{\theta} \quad (\kappa, A > 0, ~ 0<\theta<1 ~ \textrm{constants}, ~ x \geq 1 ~ \textrm{arbitrary}).
\end{align}
\end{definition}

It is clear that under Axiom A the Beurling zeta function
\begin{equation}\label{zetadef}
\ze(s):=\ze_{\G}(s):=\int_1^{\infty} x^{-s} d\N(x) = \sum_{g\in\G} \frac{1}{|g|^s}
\end{equation}
admits a meromorphic, essentially analytic continuation
$\kappa\frac{1}{s-1}+\int_1^{\infty} x^{-s} d\R(x)$ up to $\Re s >\theta$
with only one, simple pole at 1.


The Beurling zeta function \eqref{zetadef} can be used to express
the generalized von Mangoldt function
\begin{equation}\label{vonMangoldtLambda}
\Lambda (g):=\Lambda_{\G}(g):=\begin{cases} \log|p| \quad \textrm{if}\quad g=p^k,
~ k\in\NN ~~\textrm{with some prime}~~ p\in\G\\
0 \quad \textrm{if}\quad g\in\G ~~\textrm{is not a prime power in} ~~\G
\end{cases}
\end{equation}
as coefficients of the logarithmic derivative of the zeta function
\begin{equation}\label{zetalogder}
-\frac{\zeta'}{\zeta}(s) = \sum_{g\in \G} \frac{\Lambda(g)}{|g|^s}.
\end{equation}

The Beurling theory of generalized primes is mainly concerned with
the analysis of the summatory function
\begin{equation}\label{psidef}
\psi(x):=\psi_{\G}(x):=\sum_{g\in \G,~|g|\leq x} \Lambda (g).
\end{equation}

Apart from generality and applicability to e.g. distribution of
prime ideals in number fields, the interest in these things were
greatly boosted by a construction of Diamond, Montgomery
and Vorhauer \cite{DMV}. They basically showed that under Axiom A
the Riemann hypothesis may still fail: moreover, nothing better
than the most classical zero-free region and error term of
\begin{equation}\label{classicalzerofree}
\zeta(s) \ne 0 \qquad \text{whenever}~~~ s=\sigma+it, ~~ \sigma >
1-\frac{c}{\log t},
\end{equation}
and
\begin{equation}\label{classicalerrorterm}
\psi(x)=x +O(x\exp(-c\sqrt{\log x})
\end{equation}
follows from \eqref{Athetacondi} at least if $\theta>1/2$.

Therefore, Vinogradov mean value theorems on trigonometric sums
and many other stuff are certainly irrelevant in this generality,
and for Beurling zeta functions a careful revival of the
combination of "ancient-classical" methods and "elementary"
arguments can only be implemented.

In the classical case of prime number distribution, as well as regarding some extensions to primes in arithmetical progressions and distribution of prime ideals in algebraic number fields, the connection between location and distribution of zeta-zeroes and oscillatory behavior in the remainder term of the prime number formula $\psi(x)\thicksim x$ is well understood. On the other hand in the generality of Beurling primes and zeta function, investigations so far were focused on mainly three directions. First, better and better, minimal conditions were sought in order to have a Chebyshev type formula $x\ll \psi(x) \ll x$, see e.g. \cite{Vindas12, DZ-13-2, DZ-13-3}. Understandably, as in the classical case, this relation requires only an analysis of the $\zeta$ function of Beurling in, and on the boundary of the convergence halfplane. Second, conditions for the prime number theorem to hold, were sought see e.g. \cite{Beur, K-98, DebruyneVindas-PNT, DZ-17,  {Zhang15-IJM}}. Again, this relies on the boundary behavior of $\zeta$ on the one-line $\si=1$. Third, rough (as compared to our knowledge in the prime number case) estimates and equivalences were worked out in the analysis of the connection between $\zeta$-zero distribution and error term behavior for $\psi(x)$ see e.g. \cite{H-5}, \cite{H-20}. Further, examples were constructed for arithmetical semigroups with very "regular" (such as satisfying the Riemann hypothesis RH and error estimates $\psi(x)=x+O(x^{1/2+\varepsilon})$) and very "irregular" (such as having no better zero-free regions than \eqref{classicalzerofree} and no better asymptotic error estimates than \eqref{classicalerrorterm}) behavior and zero- or prime distribution, see e.g. \cite{DMV}, \cite{Zhang7}, \cite{H-15}, \cite{BrouckeDebruyneVindas}. Here we must point out that the above citations are just examples, and are far from being a complete description of the otherwise formidable literature\footnote{E.g. a natural, but somewhat different direction, going back to Beurling himself, is the study of analogous questions in case the assumption of Axiom A is weakened to e.g. an asymptotic condition on $N(x)$ with a product of $x$ and a sum of powers of $\log x$, or sum of powers of $\log x$ perturbed by almost periodic polynomials in $\log x$, or $N(x)-cx$ periodic, see \cite{Beur}, \cite{Zhang93}, \cite{H-12}, \cite{RevB}.}. Throughout analysis of these directions as well as for much more information the reader may consult the monograph \cite{DZ-16}.

With this work we start a systematic analysis of the connection between these two aspects of a Beurling generalized number system or arithmetical semigroup. As starting point, here we work out a number of relatively standard initial estimates on the behavior of the Beurling zeta function $\zs$ and its zeroes, assuming only the above mentioned Axiom A. Although these results bring no surprises, and are known for more than a century for the classical Riemann zeta function, it seems that until now there is no systematic presentation of them for the general Beurling case. Most probably some, if not all, can be found in some form or in another in various papers, but for our long-range goal with this study it seemed to be rather inconvenient to try to reference out, one by one, in different forms and various versions, all these technical auxiliary results. Moreover, we put an emphasis on getting explicit, relatively sharp constants, so that in any further research, in particular if explicit, numerical results would be needed, the future researcher could draw from our efforts. With the single assumption formulated in Axiom A, we arrive at a Riemann-von Mangoldt type formula, which, in principle, explains very clearly, even if not explicitly, the connection between the location of zeta-zeros and the oscillation of the prime number formula. From here, one can better
explain the expectations, with which we consider that certain assumptions on one side correspond to respective properties on the other side. In particular, we can at least start to think of the question of Littlewood \cite{Littlewood}, placed into the general context of the Beurling theory: what amount of oscillation is "caused by having a certain zero" $\rho=\beta+i\gamma$ of $\zeta$, and how this oscillation can be demonstrated effectively, i.e. with explicit bounds approaching the expected order, suggested by the Riemann-von Mangoldt formula? This question, and to find some explicit bounds on $\psi(x)$ exploiting "the effect of that one hypothetical zeta-zero", was a notable problem first answered by Turán \cite{Turan1}. In many respects, the best (and essentially optimal) results in this direction are due to Pintz \cite{Pintz1} \cite{PintzProcStekl}, but various ramifications, equivalences and variants were proved by a number of people including the author \cite{RevAA}. Here we refrain from any further discussion of these related number theory developments.

The present paper is not a stand-alone paper. We intend to work out a detailed analysis of the connection between distribution of $\zeta$-zeros and oscillation of the (von Mangoldt type modified) prime counting function $\psi(x)$. In the continuation of this work \cite{Rev20} we present several definitely new results on the distribution of the zeros of the Beurling zeta function. The interested reader may have a look into \cite{Rev20} to see what our direct aims are in this regard. Then, based on the analysis in these two papers dealing with the Beurling zeta function only, we intend to turn to the analysis of oscillatory properties of $\psi(x)$ itself. A more detailed description of our "number theory" results we postpone to the introduction of these forthcoming papers, but let us emphasize that the calculus here, however cumbersome, is not totally pointless.

In particular, the concluding result of this work is a Riemann-von Mangoldt type formula. The Riemann-von Mangoldt formula is  expressing oscillation of the remainder term in the prime number theorem PNT (here for Beurling primes) by suitable partial sums of a generally divergent, yet for our purpose very intuitive series over the zeroes of $\zeta$. Although the proof of this is standard, known in the classical case from the XIXth century, it seems that authors on the Beurling generalized number systems discarded this formula, possibly because of the uncontrollable divergence properties of the series. Yet it is very intuitive if we try to read down the exact correspondence between the oscillation of the remainder term of the PNT on the one hand and the location and distribution of the zeros of the Beurling zeta function on the other hand. Although the results proved in the Beurling case did not reach the generality and precision which would be suggested by these formulae, we can point out that in the classical case, and in fact in somewhat wider generality, rather precise correspondences were already proved \cite{Pintz1}, \cite{PintzProcStekl}, \cite{RevAA}. Starting with this formula thus seems to be the right first step to work out similarly sharp correspondences in the generality of Beurling number systems.

\section{Basic properties of the Beurling $\zeta$}
\label{sec:basics}

The following basic lemmas are slightly elaborated forms of 4.2.6.
Proposition, 4.2.8. Proposition and 4.2.10. Corollary of
\cite{Knopf}. These form the most basic facts concerning the
Beurling zeta function. Still, we give a proof for explicit
handling of the arising constants in these estimates.
\begin{lemma}\label{l:zetaxs} Denote the "partial sums" (partial
Laplace transforms) of $\N|_{[1,X]}$ as $\ze_X$ for arbitrary $X\geq 1$:
\begin{equation}\label{zexdef}
\ze_X(s):=\int_1^X x^{-s} d\N(x).
\end{equation}
Then $\ze_X(s)$ is an entire function and for $\sigma:=\Re s
>\theta$ it admits
\begin{equation}\label{zxrewritten}
\ze_X(s)=
\begin{cases} \ze(s)-\frac{\kappa X^{1-s}}{s-1}-\int_X^\infty x^{-s}d\R(x)
& \textrm{for all} ~~~s\ne 1,  \\
\frac{\kappa }{s-1}-\frac{\kappa X^{1-s}}{s-1}+\int_1^X
x^{-s}d\R(x)
& \textrm{for all} ~~~s\ne 1, \\
\kappa \log X + \int_1^X \frac{d\R(x)}{x} & \textrm{for}~~~ s=1,
\end{cases}
\end{equation}
together with the estimate
\begin{equation}\label{zxesti}
\left|\ze_X(s) \right| \leq \ze_X(\sigma) \leq
\begin{cases} \min \left( \frac{\kappa X^{1-\sigma}}{1-\sigma}
+ \frac{A}{\sigma-\theta},~ \kappa X^{1-\sigma}\log X +
\frac{A}{\sigma-\theta}\right) &\textrm{if} \quad
\theta<\sigma < 1,
\\ \kappa \log X + \frac{A}{1-\theta}& \textrm{if} \qquad \sigma =1,
\\ \min\left( \frac{\sigma (A+\kappa)}{\sigma-1},~{\kappa}\log X + \frac{\sigma A}{\sigma-\theta}
\right) &\textrm{if}\quad
\sigma>1.
\end{cases}
\end{equation}
Moreover, the above remainder terms can be bounded as follows.
\begin{equation}\label{zxrlarge}
\left| \int_X^\infty x^{-s}d\R(x) \right| \leq A
\frac{|s|+\sigma-\theta}{\sigma-\theta} X^{\theta-\sigma}
\end{equation}
and
\begin{equation}\label{zxrlow}
\left| \int_1^X x^{-s}d\R(x) \right| \leq A \left(
|s|\frac{1-X^{\theta-\sigma}}{\sigma-\theta} +
X^{\theta-\sigma}\right) \leq A \min \left(
\frac{|s|}{\sigma-\theta},~|s| \log X + X^{\theta-\sigma} \right).
\end{equation}
\end{lemma}
\begin{proof} Clearly, both $\ze_X$ and $\int_1^X x^{-s} d\R(x)$ are
entire functions, $\ze$ is regular for $\Re s > \theta$ except at
$s=1$, $X^{1-s}/(s-1)$ is regular all over $\CC\setminus \{1\}$,
while $\int_X^\infty x^{-s} d\R(x)$ is regular for $\Re s
>\theta$. Hence it suffices to prove the formulae \eqref{zxrewritten}
for $\Re s>1$, and then refer to analytic continuation in
extending them all over the common domain of regularity of both
sides, i.e. $\{ \Re s > \theta \} \setminus \{1\}$.

The formulae then follow as for $\Re s > 1$ the defining integrals
are absolutely convergent, and to get the first form we only need
to recall $\N(x)=\kappa (x-1) +\R(x)$ and write
$$
\ze_X(s)=\ze(s)-\int_X^\infty x^{-s}d\N(x) = \ze(s) - \kappa
\int_X^\infty x^{-s} d(x-1) - \int_X^\infty x^{-s} d\R(x);
$$
while to obtain the second and also the third ones it suffices to
write again $\N(x)=\kappa (x-1) +\R(x)$ and execute the integration
directly for the main term.

To prove the estimate \eqref{zxesti} we utilize that by definition
$\N(x)$ is increasing, hence $d\N(x)$ is positive, and thus
$|\ze_X(s)|\leq \ze_X(\sigma)$, which in turn can be calculated by
partial integration and substituting \eqref{Athetacondi}. Namely,
for $\sigma\ne 1$ and $\sigma>\theta$
\begin{align*}
\ze_X(\sigma) &= \left[ x^{-\sigma} \N(x)\right]_1^X
+\sigma\int_1^X x^{-\sigma-1} \N(x) dx
\\ & \leq \frac{\kappa (X-1) + A X^{\theta}}{X^{\sigma}} + \sigma \int_1^X
\frac{\kappa(x-1)}{x^{1+\sigma}}dx +\sigma \int_1^X
\frac{A}{x^{\sigma-\theta+1}}dx
\\ & = \kappa (X^{1-\sigma}-X^{-\si}) + A X^{\theta-\sigma} + \frac{\sigma
\kappa}{1-\sigma} \left(X^{1-\sigma} -1 \right) + \kappa\left(X^{-\sigma} -1 \right) + \frac{\sigma A}{\sigma-\theta} \left( 1-X^{\theta-\sigma}\right)
\\ & = \frac{\kappa}{1-\sigma} \left(X^{1-\sigma} -1 \right) +
\frac{\sigma A}{\sigma-\theta} -\frac{\theta A}{\sigma-\theta} X^{\theta-\sigma}
< \frac{\kappa}{1-\sigma} \left(X^{1-\sigma} -1 \right) +
\frac{\sigma A}{\sigma-\theta}
\\ & \leq \begin{cases} \min \left( \frac{\kappa X^{1-\sigma}}{1-\sigma}
+ \frac{A}{\sigma-\theta},~ \kappa X^{1-\sigma}\log X +
\frac{A}{\sigma-\theta}\right) &\textrm{if} \quad \sigma<1 \\
\min\left( \frac{\sigma (A+\kappa)}{\sigma-1},~ {\kappa}\log X +
\frac{\sigma A}{\sigma-\theta} \right) &\textrm{if}\quad
\sigma>1
\end{cases},
\end{align*}
with an application of $\frac{e^y-1}{y}\le e^y\frac{1-e^{-y}}{y}\le e^y ~(y>0)$ in the very
last estimations involving $\log X$. The case $\sigma=1$ follows
by taking the limit of the second estimate in the first line of \eqref{zxesti} when $\sigma\to 1-$, as $\ze_X(\sigma)$ is regular.

It remains to estimate the error terms arising from $\R(x)$. After
partial integration
$$
\int x^{-s} d\R(x) = \left[\R(x) x^{-s} \right] +s \int
\frac{\R(x)}{x^{s+1}} dx
$$
and computing the actual values of the integrated part and
applying $|\R(x)|\leq Ax^\theta$ both there and in the integrals
lead to \eqref{zxrlarge} and the first estimate in \eqref{zxrlow},
resp. The last estimate in \eqref{zxrlow} follows again by
$\frac{1-e^{-y}}{y}\le 1 ~(y>0)$.
\end{proof}

\begin{lemma}\label{l:zkiss} We have
\begin{equation}\label{zsgeneral}
\left|\zs-\frac{\kappa}{s-1}\right|\leq \frac{A|s|}{\sigma-\theta}
\qquad \qquad\qquad (\theta <\sigma ,~ t\in\RR,~~ s\ne 1).
\end{equation}
In particular, for large enough values of $t$ it holds
\begin{equation}\label{zsgenlarget}
\left|\zs \right|\leq \sqrt{2} \frac{(A+\kappa)|t|}{\sigma-\theta}
\qquad \qquad\qquad (\theta <\sigma \leq |t|).
\end{equation}
while for small values of $t$ we have
\begin{equation}\label{zssmin1}
|\zs(s-1)-\kappa|\leq \frac{A|s||s-1|}{\sigma-\theta} \leq
\frac{100 A}{\sigma-\theta}\qquad (\theta <\sigma \leq 4,~|t|\leq
9).
\end{equation}
As a consequence, we also have
\begin{equation}\label{polenozero}
\zs\ne 0 \qquad \textrm{for}\qquad  |s-1| \leq
\frac{\kappa(1-\theta)}{A+\kappa}.
\end{equation}
\end{lemma}
\begin{proof}
Actually, by $\N(x)=\kappa x +\R(x)$ and the trivial integral
$\int_1^{\infty} x^{-s}dx=1/(s-1)$, to get \eqref{zsgeneral} it remains only to see that
$|\int_1^{\infty} x^{-s}d\R(x)| \leq A|s|/(\sigma-\theta)$. That
is straightforward (and is contained, after $X\to\infty$, in
\eqref{zxrlow}, too). Whence we have \eqref{zsgeneral}.

To deduce \eqref{zsgenlarget} one has to take into account that in
the given domain we have $t^2+\sigma^2 \leq 2 t^2$, hence $|s|\leq
\sqrt{2}|t|$, and that $t^2 |s-1|^2 /(\sigma-\theta)^2 =
t^2(t^2+(\si-1)^2)/(\sigma-\theta)^2 > \si^2(\si^2+(\si-1)^2)/\si^2 = \si^2+(\si-1)^2 \ge 1/2$, hence $1/|s-1| \le \sqrt{2} |t|/(\sigma-\theta)$, too.

To obtain \eqref{zssmin1} from \eqref{zsgeneral} it suffices to estimate the arising factor $|s||s-1|$ trivially.

Further, from \eqref{zssmin1} it follows that to have that $\zs (s-1) = \kappa + (s-1) \int_1^{\infty} x^{-s}d\R(x)=\kappa + \left(\zs(s-1)-\kappa\right)\ne 0$, it suffices that $ \kappa > |s(s-1)| \frac{A}{\sigma-\theta}$, that is $\frac{\kappa(\si-\theta)}{A|s|} > |s-1|$.

Note that $\zs \ne 0$ for $\si>0$, whence we can restrict considerations to the critical strip $\theta<\si\le 1$. Let now $0<r<1-\theta$ be a parameter, to be specified later, and assume that $|s-1|\le r$. Then we are to show $f(s):=\frac{\kappa(\si-\theta)}{A|s|}>r$ for $s=\si+it \in \{s~:~ \theta<\si\le 1, |s-1|\le r\}=:G$. It is clear that for any fixed value of $\si=\si_0$, $f(\si_0+it)$ becomes minimal at the endpoints $\si_0\pm it_0$, where $t_0=\sqrt{r^2-(1-\si_0)^2}$. Further, the function value of $f$ is symmetric with respect to the real axis, so in fact it suffices to take into account the upper endpoint here. So consider the function $f$ restricted to these points on the quarter circle bounding $G$ from the upper left side, i.e. $f(1+r\cos \ff+ir\sin\ff)$, as $\pi/2\le\ff\le \pi$. We have
$f(1+r\cos\ff+ir\sin\ff) =\frac{\kappa}{A} \frac{1+r\cos\ff -\theta}{\sqrt{(1+r\cos\ff)^2+(r\sin\ff)^2}}=\frac{\kappa}{A} \sqrt{\frac{(1+r\cos\ff -\theta)^2}{(1+r\cos\ff)^2+(r\sin\ff)^2}}=\frac{\kappa}{A} \sqrt{\frac{(1-\theta-r\lambda)^2}{1-2r\lambda+r^2}}$, where we have put $\lambda:=-\cos\ff \in [0,1]$. Whence it suffices to find the minimum of the function $g(\lambda):=\frac{(1-\theta-r\lambda)^2}{1-2r\lambda+r^2}$. A little calculus shows that $g'(\lambda)<0$, the function is strictly decreasing, whence its minimum is attained at the point $\lambda=1$. Winding up, this means that $f$ is minimal at the point $s=1-r$. It remains to derive $f(1-r)= \frac{\kappa}{A} \frac{1-\theta-r}{1-r} >r$ for a sufficiently small choice of the parameter $r$. This is equivalent to $\frac{\kappa(1-\theta)}{A} > \frac{\kappa+A}{A}r-r^2$, which surely holds whenever $\frac{\kappa(1-\theta)}{A} \ge \frac{\kappa+A}{A}r$, i.e. whenever $r\le \frac{\kappa(1-\theta)}{\kappa+A}$, as needed.

\end{proof}

\begin{lemma}\label{l:oneperzeta} We have
\begin{equation}\label{zsintheright}
|\zs| \leq \frac{(A+\kappa)\sigma}{\sigma-1} \qquad (\sigma >1).
\end{equation}
and also
\begin{equation}\label{reciprok}
|\zs| \geq \frac{1}{\ze(\sigma)} >
\frac{\sigma-1}{(A+\kappa)\sigma} \qquad (\sigma >1).
\end{equation}
\end{lemma}
\begin{proof} The first estimate follows from \eqref{zxesti}, last line after
$X\to\infty$. To obtain the second, consider the Dirichlet series
of $1/\ze$: as $|\mu(g)|\leq 1$, the coefficientwise estimation of
$1/\ze$ provides the estimate $|1/\zs|\leq \ze(\sigma)$, and the
rest is already contained in the first inequality.
\end{proof}

\begin{lemma}\label{l:zetainrighthp} In the right halfplane
of convergence we have
\begin{equation}\label{zsintherightlarget}
|\zs| \leq \frac{\kappa}{\sigma-\theta}\log|t| + \frac73
\frac{\sigma(A+\kappa)}{\sigma-\theta} \qquad (\sigma
>1,~|t|\geq 4)
\end{equation}
\end{lemma}
\begin{proof} By the first formula in \eqref{zxrewritten} we can write
\begin{equation}\label{zetazetax}
\zs=\ze_X(s)+\frac{\kappa
X^{1-s}}{s-1}+\int_X^{\infty}\frac{d\R(x)}{x^s}.
\end{equation}
From here by termwise estimation, using the second part of the
third line in \eqref{zxesti} in combination with \eqref{zxrlarge}, we arrive at
$$
|\zs|\leq {\kappa}\log X + \frac{\sigma A}{\sigma-\theta}
+ \frac{\kappa}{|t|} X^{1-\sigma} +
\frac{A(|s|+\sigma)}{\sigma-\theta} X^{\theta-\sigma}.
$$
Since $|t|\geq 4$ and $1<\sigma$, we have
$|s|/(|t|\sigma)=\sqrt{|t|^{-2}+\sigma^{-2}} \leq (33/32)$, and
$1/|s-1|\leq 1/4 \leq \sigma/(\sigma-\theta)$, so for
$X:=|t|^{1/(\sigma-\theta)}\geq 1$ we also have
$$
|\zs|\leq \frac{\kappa}{\sigma-\theta}\log |t| +
\frac{\sigma  A}{\sigma-\theta} + \frac{\kappa}{4} +
\frac{A\left(\frac{33}{32}|t|\sigma+\sigma\frac{|t|}{4}\right)}{\sigma-\theta}
\frac{1}{|t|} < \frac{\kappa}{\sigma-\theta}\log |t| +
\frac{7\sigma(\kappa+ A)}{3(\sigma-\theta)}.
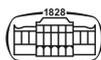
$$
\end{proof}

For smaller values of $t$, but still separated from zero, \eqref{zsintherightlarget} can be complemented by a similar estimate not containing any term with $\log |t|$, but for $t$ very close to $0$ one needs to take into account the singularity of $\ze$ at $1$, so that the best we can have is an estimate of the form \eqref{zsgeneral}. Here is an explicit formulation of these, actually valid for all $\si>\theta$.

\begin{lemma}\label{r:zestismallt} Let $\tau \ge 1$ be any parameter. Then we have
\begin{equation}\label{zestitnearzero}
\left| \zs - \dfrac{\kappa}{s-1} \right|\le (\tau+1) \dfrac{A\max(1,\si)}{\si-\theta} \qquad  (\si>\theta, |t|\le \tau).
\end{equation}
and
\begin{equation}\label{zestitsmall}
|\ze(s)| \le \dfrac{(\tau+1)(A+\kappa) \max(1,\si) }{\si-\theta}  \qquad (\si>\theta, ~ \frac{1}{\tau+1} \le |t|\le \tau).
\end{equation}
Furthermore, we also have the estimate 
\begin{equation}\label{zestitnotlarge}
|\ze(s)|\le (\tau+1)(A+\kappa) \max(1,\si) \max\left(\frac{1}{\si-\theta}, \frac{1}{|1-\si|}\right) \qquad  (\si>\theta, |t|\le \tau).
\end{equation}
\end{lemma}

\begin{proof} The estimate \eqref{zestitnearzero} follows directly from \eqref{zsgeneral} and $|s|\le 
\si+\tau \le (\tau+1)\max(1,\si)$.

To see \eqref{zestitsmall} we need to note only $1/|s-1| \le 1/|t| \le \tau+1 \le (\tau+1)\si/(\si-\theta)$ and take into account \eqref{zestitnearzero}, already proven.

The assertion \eqref{zestitnotlarge} is a direct consequence of \eqref{zestitnearzero} if we take into account the trivial estimate $\kappa/|s-1|\le \kappa/|\si-1|$.
\end{proof}

\begin{lemma}\label{l:zetaincriticalstrip}
In the critical strip we have
\begin{equation}\label{zsintheleft}
|\zs| \leq 2~ (A+\kappa) \max \left( \frac{1}{\sigma-\theta},
\frac{1}{1-\sigma}\right)\cdot |t|^{\frac{1-\sigma}{1-\theta}}
\qquad (\theta <\sigma <1,~ |t| \geq 4).
\end{equation}
Moreover, we also have
\begin{equation}\label{zslogos}
|\zs| \leq 2 \frac{A+\kappa}{\sigma-\theta}
|t|^{\frac{1-\sigma}{1-\theta}} \cdot \log |t| \qquad (\theta
<\sigma <1,~ |t| \geq e^{5/4}).
\end{equation}
and
\begin{equation}\label{zslogoskist}
|\zs| \leq \frac52 \frac{A+\kappa}{\sigma-\theta}
|t|^{\frac{1-\sigma}{1-\theta}} \qquad (\theta
<\sigma <1,~ 1 \leq |t| \leq e^{5/4}).
\end{equation}
\end{lemma}
Note that by \eqref{zestitnotlarge} 
we have $|\zs| \le 2(A+\kappa) \max \left( \frac{1}{\sigma-\theta}, \frac{1}{|1-\sigma|}\right)$ for $\theta<\si<1, ~|t| \le 1$.

\begin{proof} Again we start with \eqref{zetazetax}, coming from
the first formula in \eqref{zxrewritten}. Since now $|t|\geq 4$
and $\theta<\sigma<1$, we have $|s|\leq (5/4)|t|$,
$|s|+\sigma-\theta < 3/2|t|$, and $1/|s-1|\leq 1/4$, so the
estimates of the first part of the first line of \eqref{zxesti}
and \eqref{zxrlarge} yield whenever $X\geq 1$ that
$$
|\zs|\leq \frac{5}{4} \frac{\kappa X^{1-\sigma}}{1-\sigma} +
\frac{A}{\sigma-\theta} \left(1+\frac32 |t| X^{\theta-\sigma}
\right).
$$
First consider the case when $\theta<\si<\frac{1+\theta}{2}$. In this case we have $1-\si>\frac12(1-\theta)$, whence in view of $|t|\ge 4$ also $|t|^{\frac{1-\sigma}{1-\theta}}\geq
4^{1/2}$ allowing to estimate 1 by $\frac12
|t|^{\frac{1-\sigma}{1-\theta}}$. If on the other hand $\frac{1+\theta}{2}\le \si <1$, then we have $\frac{1}{\si-\theta}\le\frac12 \frac{1}{1-\si}|t|^{\frac{1-\sigma}{1-\theta}}$, because $\frac{\sigma-\theta}{1-\sigma}|t|^{\frac{1-\sigma}{1-\theta}}\ge \frac{\sigma-\theta}{1-\sigma}4^{\frac{1-\sigma}{1-\theta}}=
\left(\frac{1-\theta}{1-\sigma}-1\right)4^{\frac{1-\sigma}{1-\theta}}$
and the function $\varphi(u):=(1/u-1)4^u$ is decreasing in $(0,1/2)$ giving $\varphi\left(\frac{1-\theta}{1-\sigma}\right)\ge \varphi(1/2)=4^{1/2}=2$. In sum, we obtain for both cases
$$
|\zs|\leq \left(\frac{5}{4} \kappa X^{1-\sigma} +
2A |t| X^{\theta-\sigma}
\right) \max\left(\frac{1}{1-\sigma}, \frac{1}{\sigma-\theta}\right) .
$$
So, choosing $X:=|t|^{1/(1-\theta)}\geq 1$ implies the first
estimate in \eqref{zsintheleft}.

To get also the estimates \eqref{zslogos} and \eqref{zslogoskist}, we use the second
part of the first line of \eqref{zxesti} together with
\eqref{zxrlarge}. Similarly as above we again choose
$X:=|t|^{1/(1-\theta)}\geq 1$ and obtain
\begin{align*}
|\zs|& \leq \left(\kappa+\kappa X^{1-\sigma} \log X +
\frac{A}{\sigma-\theta}\right) + \frac{\kappa}{4}X^{1-\sigma}
+ \frac32\frac{A}{\sigma-\theta}|t|X^{\theta-\sigma} \\
& \leq \frac54\kappa |t|^{\frac{1-\sigma}{1-\theta}} +
\frac{\kappa}{1-\theta} |t|^{\frac{1-\sigma}{1-\theta}}\log|t|+
\frac52 \frac{A}{\sigma-\theta} |t|^{\frac{1-\sigma}{1-\theta}} \le \left(\kappa \log|t| + \frac54 \kappa +\frac52 A \right) \frac{|t|^{\frac{1-\sigma}{1-\theta}}}{\sigma-\theta}.
\end{align*}
From here in case $|t|\ge e^{5/4}$ we find $5/4 \le \log|t|$ whence $5/4 ~\kappa\le \kappa \log|t|$ and $5/2 ~A \le 2 \log|t|$, resulting in \eqref{zslogos}. In case $1\le |t|\le e^{5/4}$ using $0\le \log|t|\le 5/4$ directly gives \eqref{zslogoskist}.
\end{proof}

\section{Estimates for the number of zeros of  $\zeta$}\label{sec:zeros}

In this section we work out explicit estimates for the number of $\zeta$-zeroes in various regions. The underlying principle is the well-known fact that if the zeta function has finite order, then Jensen's inequality always leads to a local bound $\ll \log |t|$ as below.

\begin{lemma}\label{l:Jensen} Let $\theta<b<1$ and consider
the rectangle $H:=[b,1]\times [i(T-h),i(T+h)]$, where $h:=\frac{\sqrt{7}}{3}
\sqrt{(b-\theta)(1-\theta)}$ and $|T| \ge e^{5/4}+\sqrt{3}\approx 5.222...$ is arbitrary.

Then the number $n(H)$ of zeta-zeros in the rectangle $H$ satisfy
\begin{align}\label{zerosinH}
n(H) & \leq \frac{1-\theta}{b-\theta}\left(0.654 \log|T| + \log\log |T| + 6\log(A+\kappa) + 6\log\frac1{1-\theta} +12.5 \right)
\notag \\ &\leq \frac{1-\theta}{b-\theta}\left(\log|T| + 6\log(A+\kappa) + 6\log\frac1{1-\theta} +12.5 \right)
\end{align}

Moreover, if $|T|\le 5.222$, then we have analogously the $\log|T|$-free estimate
\begin{equation}\label{zerosinH-smallt}
n(H) \leq \frac{1-\theta}{b-\theta}\left(6\log(A+\kappa) + 6\log\frac1{1-\theta}+14\right).
\end{equation}
\end{lemma}
\begin{remark}\label{r:alsoindisk} Note that our estimate includes also the total number $N$ of zeroes in the disc $\mathcal{D}_r:=\{s~:~|s-(p+iT)|\leq r:=p-q\}$, where $p:=1+(1-\theta)$ and
$q:=\theta+\frac23(b-\theta)$ are parameters introduced in the proof.
\end{remark}
\begin{proof} Given $b$, let us introduce the parameters $a,p$ and $q$ satisfying $\theta<a<q<b<1<p\leq 2$.
We will choose the concrete values of these parameters a few lines below.

Let us draw the circles ${\mathcal C}_R$ of radius $R:=p-a$ and ${\mathcal C}_r$
of radius $r:=p-q$ around $z_0:=p+iT$. Then the disk ${\mathcal
D}_r$ bounded by ${\mathcal C}_r$ will cover the rectangle
$H_1:=[b,1]\times [T-ih_1,T+ih_1]$ with $h_1:=\sqrt{(b-q)(2p-b-q)}$.
Actually, in the following we estimate the total number $N$ of roots
in ${\mathcal D}_r$: since $H_1\subset {\mathcal D}_r$, it follows
that also the number $n_1$ of roots situated in $H_1$ satisfy this estimate.

The proof will then be concluded by showing that our choice $p:=1+(1-\theta)$ and
$q:=\theta+\frac23(b-\theta)$ of parameters imply $h\le h_1$, i.e. $H\subset H_1$,
so that $n(H) \le n_1\le N$, too. The parameter $a$ will be chosen to be arbitrarily
close to $\theta$, i.e. $a\to\theta+$.

We use Jensen's inequality (see \cite[p. 43]{Hol}) in the form
that for $f$ regular in $D(z_0,R)$ and for any $r<R$
$$
\log|f(z_0)| + \nu \log\frac Rr \leq \frac{1}{2\pi}
\int_{-\pi}^{\pi} \log|f(z_0+R e^{i\varphi})|d\varphi,
$$
where $\nu$ is the number of zeroes of $f$ in the disk $D(z_0,r)$.

To apply this, now we put $z_0:=p+iT$,
$R:=p-a(<2-\theta \leq 2)$, $r:=p-q$, and $f:=\zeta$.
Then $\zeta$ is indeed regular in the disc $D(z_0,R)$,
and Jensen's inequality yields
$$
N \log \frac{R}{r} \leq \frac{1}{2\pi}
\int_{-\pi}^{\pi} \log|\zeta(z_0+R e^{i\varphi})|d\varphi - \log|\zeta(p+iT)|,
$$
with $N$ denoting he number of $\ze$-zeroes in $D(z_0,r)={\mathcal D}_r$.

According to Lemma \ref{l:oneperzeta}, formula \eqref{reciprok} we have $\log|\ze(p+iT)| \geq \log\frac{p-1}{(A+\kappa)p}$,
thus using also $\log R/r = -\log r/R =-\log\left(1-\frac{R-r}{R}\right)> \frac{R-r}{R}=\frac{q-a}{p-a}$
we are led to
\begin{equation}\label{Ninthedisk}
N \frac{q-a}{p-a} \leq
\frac{1}{2\pi} \int_{-\pi}^{\pi} \log|\zeta(z_0+R
e^{i\varphi})|d\varphi + \log \frac{(A+\kappa)p}{p-1}.
\end{equation}

It remains to estimate the integral. We will cut the integral into two parts according to $s=\si+it=z_0+Re^{i\ff}$ belonging to the right halfplane of convergence or to the critical strip. The first case occurs when $\si=p+R\cos\ff>1$ and the second when $\si<1$, so that the first case happens exactly for $|\ff|<\alpha:=\arccos\left(-\frac{p-1}{R}\right)=\arccos\left(-\frac{1-\theta}{R}\right)$, and the second when $\alpha<|\ff|<\pi$ (the one point equality cases being irrelevant for the evaluation of the integrals in question).

\bigskip
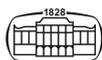
\underline{Part 1 ($\si>1$).} Then we have according to the uniform estimate \eqref{zsintheright} of Lemma \ref{l:oneperzeta} that
$$
\log|\zeta(z_0+Re^{i\varphi})| \le \log(A+\kappa)+\log\frac{\si}{\si-1} \le \log(A+\kappa)+(\si-1) +\log\frac{1}{\si-1}.
$$
Thus for the relevant part of the integral we obtain
\begin{align*}
\int_{-\al}^{\al} \log|\zeta(z_0+Re^{i\varphi})|d\varphi &\le 2\int_0^\al \left(\log(A+\kappa)+(p+R\cos\ff-1) + \log\frac{1}{R} + \log\frac{R}{p-1+R\cos\ff} \right)d\ff
\\&= 2\al\left(\log(A+\kappa)+(1-\theta)+ 2R\sin\al + \log\frac{1}{R} \right) + 2\int_0^\al \log\frac{1}{\frac{p-1}{R}+\cos\ff}d\ff.
\end{align*}
So writing here $w:=\frac{p-1}{R}=\frac{1-\theta}{R}>1/2$ (the value of which is to converge to $1/2+0$ finally), and recalling that $\al=\arccos(-w)$, we now estimate the last integral as follows.
\begin{align*}
\int_0^\al \log\frac{1}{w+\cos\ff}d\ff& =\int_{-w}^1 \log\frac{1}{w+u}\frac{du}{\sqrt{1-u^2}}
\\& = \left[(\arcsin u- \arcsin(-w))\log\frac{1}{w+u} \right]_{-w}^1 - \int_{-w}^1 (\arcsin u- \arcsin(-w)) \frac{-1}{w+u}   du
\\& = \left(\frac{\pi}{2}-(\frac{\pi}{2}-\alpha)\right) \log\frac{1}{w+1}-0 +\int_{-w}^1 \frac{\arcsin u- \arcsin(-w)}{u-(-w)} du.
\end{align*}
When $u$ increases from $-w\approx-1/2$ to $0$, the difference quotient $\frac{\arcsin u- \arcsin(-w)}{u-(-w)}$ under the integral sign is only decreases (for $\arcsin$ is concave on $[-w,0]$), and then it starts to increase again (for convexity of $\arcsin$ on $[0,1]$), so that the maximum of the derivative $\arcsin'(-w)$ at $-w$ and the slope of the chord between the points $(-w,\arcsin(-w))$ and $(1,\pi/2)$ is a valid upper estimation for the integrand. That is,
$$
\frac{\arcsin u- \arcsin(-w)}{u-(-w)} \le \max \left(\frac{1}{\sqrt{1-w^2}}, \frac{\pi/2+\arcsin w}{1+w}\right)=\max \left(\frac{1}{\sqrt{1-w^2}}, \frac{\alpha}{1+w}\right).
$$
Recalling that the value of $w$ is to be close to $1/2$, and thus $\alpha\approx 2\pi/3$, we get that the last maximum is the second expression. As a result,
$$
\int_0^\al \log\frac{1}{w+\cos\ff}d\ff \le \al \log\frac{1}{w+1}+\alpha.
$$
Therefore,
\begin{align}\label{logzetaintegralinhalfplane}
\int_{-\al}^{\al} \log|& \zeta(z_0+Re^{i\varphi})|d\varphi
\le 2\al\left(\log(A+\kappa)+(2-\theta)+ 2R\sin\al + \log\frac{1}{R(w+1)}\right).
\end{align}
Note that until here we did not use any assumption on the value of $T$, it could be arbitrary.

\bigskip
\underline{Part 2: ($\theta<\si<1$)}. Along the part $\Cm$ of the circle ${\mathcal C}_R$, which belongs to the critical strip, the minimal value of $t=\Im s=\Im (T+Re^{i\ff})$ is $T-R\sin(-\alpha)=T-R\sqrt{1-(\frac{p-1}{R})^2}=T-\sqrt{R^2-(p-1)^2}>T-\sqrt{(p-\theta)^2-(p-1)^2} =T-\sqrt{3}(1-\theta)$, and the maximal value of $t=\Im s$ is $T+\sqrt{3}(1-\theta)$.

\bigskip
\underline{Case 2.1: $|T|\ge e^{5/4}+\sqrt{3}$.}
So, for $T\ge e^{5/4} +\sqrt{3}$ the whole curve $\Cm$ lies in the range where $t\ge e^{5/4}$, and the estimates of Lemma \ref{l:zetaincriticalstrip}, in particular \eqref{zslogos} apply. We can therefore write
\begin{align*}
\int_{|\ff| \ge \al} \log|& \zeta(z_0+Re^{i\varphi})|d\varphi \le \int_{|\ff| \ge \al} \log\left(2 \frac{A+\kappa}{\sigma-\theta}
|t|^{\frac{1-\sigma}{1-\theta}} \cdot \log |t|\right) d\ff
\\ &\le 2(\pi-\al)\log\left(2(A+\kappa)\right)+\int_{|\ff| \ge \al} \left( -\log(\si-\theta) + \frac{1-\sigma}{1-\theta} \log t + \log\log t \right) d\ff.
\end{align*}
Recalling $s=\si+it=p+iT+R\cos\ff+iR\sin\ff$, $R<2(1-\theta)$ and $\al>2\pi/3$ leads to
\begin{align}\label{intlogsigmamtheta}
\int_{|\ff| \ge \al}   \log(\si-\theta) d\ff & = 2 \int_\alpha^\pi \log\left(p-\theta+R\cos\ff\right)d\ff=2 (\pi-\al)\log R + 2\int_\alpha^\pi\log\left(\frac{2(1-\theta)}{R}+\cos\ff\right) d\ff \notag
\\& > 2 (\pi-\al)\log R + 2\int_{2\pi/3}^\pi\log\left(1+\cos\ff\right) d\ff
\\& = 2 (\pi-\al)\log R + \frac{2\pi}{3}\log 2 +8 \int_0^{\pi/6} \log(\sin u)du \approx 2 (\pi-\al)\log R  -5.510912076,\notag
\end{align}
at the end calculating numerically the value of the last definite integral.
Further,
\begin{align*}
& \int_{|\ff| \ge \al} \left(\frac{1-\sigma}{1-\theta} \log t + \log\log t \right) d\ff
\\& = \int_\al^\pi
\left(\frac{1-p-R\cos\ff}{1-\theta} \log (T-R\sin\ff)(T+R\sin\ff) + \log\log (T-R\sin\ff)+ \log\log (T+R\sin\ff)\right) d\ff
\\& \le \int_\al^\pi \left(\frac{1-p-R\cos\ff}{1-\theta} \log (T^2-R^2\sin^2\ff) + 2\log(\frac12 \log (T^2-R^2\sin^2\ff))\right) d\ff,
\end{align*}
using with $x:=\log(T-R\sin\ff)$ and $y:=\log(T+R\sin\ff)$ that $\log x+\log y \le 2\log\frac{x+y}{2}$
in view of concavity of the $\log$ function. Thus estimating $(T^2-R^2\sin^2\ff)$ simply by $T^2$ we are led to
\begin{align*}
\int_{|\ff| \ge \al} & \left(\frac{1-\sigma}{1-\theta} \log t + \log\log t \right) d\ff
\le 2 \log T \int_0^{\pi-\al} \left(-1+\frac{R}{1-\theta}\cos\ff \right) d\ff + 2(\pi-\al) \log\log T
\\&= 2\log T \left(-(\pi-\al)+\frac{R}{1-\theta} \sin(\pi-\al)\right)+2(\pi-\al) \log\log T.
\end{align*}
Adding the obtained estimates furnishes
\begin{align*}
\int_{|\ff| \ge \al} \log| \zeta(z_0+Re^{i\varphi})|d\varphi & \le 2(\pi-\al)\log\left(2(A+\kappa)\right) - 2 (\pi-\al)\log R + 5.511
\\& + 2\log T \left(-(\pi-\al)+\frac{R}{1-\theta} \sin \al\right)+2(\pi-\al) \log\log T.
\end{align*}
Combining this estimate with \eqref{logzetaintegralinhalfplane} yields
\begin{align*}
\frac1{2\pi} \int_{-\pi}^{\pi}  \log| \zeta(z_0+Re^{i\varphi})|d\varphi & \le \frac{\al}{\pi} \left(\log(A+\kappa)+(2-\theta)+ 2R\sin\al + \log\frac{1}{R}
+\log\frac{1}{w+1}\right)
\\&+\left(1-\frac{\al}{\pi}\right)\log\left(2(A+\kappa)\right) - \left(1-\frac{\al}{\pi}\right) \log R +1.755
\\& + \frac{1}{\pi}\log T \left(-(\pi-\al)+\frac{R}{1-\theta} \sin \al\right)+\left(1-\frac{\al}{\pi}\right) \log\log T
\\&=\log(A+\kappa) + \left(1-\frac{\al}{\pi}\right)\log2 +\frac{\al}{\pi}(2-\theta)+2R\frac{\alpha\sin\al}{\pi} + \frac{\al}{\pi}\log\frac{1}{w+1} +1.755
\\& + \log\frac{1}{R} + \left(\frac{R}{1-\theta} \frac{\sin \al}{\pi}-\left(1-\frac{\al}{\pi}\right) \right) \log T +\left(1-\frac{\al}{\pi}\right) \log\log T =:J(a)
\end{align*}
Note that with fixed $\theta$ and the given choice of parameters $p,q$, all other parameters $\al, w$ depend only on $a$.

Now we infer from this and \eqref{Ninthedisk} the estimate
$$
N\le \frac{p-a}{q-a} \left\{J(a)+ \log \frac{(A+\kappa)p}{p-1} \right\}.
$$
The left hand side is a fixed integer (the number of zeta-zeroes in ${\mathcal D}_r$), while the right hand side estimate is valid for all $a>\theta$. Allowing now $a\to\theta+$, and using that then $R\to 2(1-\theta)$, $w=\frac{p-1}{R}\to 1/2$, $\alpha=\arccos(-w)\to 2\pi/3$, we obtain
$$
N\le \frac{p-\theta}{q-\theta} \left\{J(\theta)+ \log \frac{(A+\kappa)p}{p-1} \right\},
$$
where
\begin{align*}
J(\theta)&=\log(A+\kappa) + \frac{2}{3}(2-\theta)+2(1-\theta)\frac{\sqrt{3}}{3}+ \left(\frac{1}{2}+\frac{\sqrt{3}}{2\pi}\right)\log\frac{2}{3}
+1.755
\\& -\frac{2}{3}\log2 + \log\frac{1}{1-\theta} + \left(\frac{\sqrt{3}}{\pi}-\frac{1}{3} \right) \log T +\frac{1}{3} \log\log T
\end{align*}
so that
\begin{align*}
N & \le \frac{3(1-\theta)}{b-\theta} \left\{J(\theta) + \log(A+\kappa)+\log \frac{p}{p-1}\right\}
\\&\le \frac{3(1-\theta)}{b-\theta} \bigg\{ 2 \log(A+\kappa) + 2 \log\frac{1}{1-\theta} + \left(\frac{\sqrt{3}}{\pi}-\frac{1}{3} \right) \log T +\frac{1}{3} \log\log T +C(\theta)\bigg\}
\end{align*}
with
$$
C(\theta):=\log(1+(1-\theta)) + \frac{2}{3}(2-\theta)+2(1-\theta)\frac{\sqrt{3}}{3}+ \left(\frac{1}{2}+\frac{\sqrt{3}}{2\pi}\right)\log\frac{2}{3}
+1.755 -\frac{2}{3}\log2 < 4.165 .
$$
The estimate \eqref{zerosinH} follows.

\bigskip
\underline{Case 2.2: $2 \le T \le e^{5/4}+\sqrt{3}$.}

In this case, points on ${\Cm}$ satisfy $s=\si+it,~|t|\le e^{5/4}+2\sqrt{3} \approx 6.9544 <7=:\tau$ and also $|t|\ge 2-\sqrt{3}\theta - \sqrt{3}(1-\theta)= 2-\sqrt{3} >0.125=1/8=1/(\tau+1)$, whence all over $\Cm$ we can use the estimate of \eqref{zestitsmall}, providing for the respective integral
\begin{align}\label{intoflogzetainstrip}
\int_{|\ff| \ge \al} \log| \zeta(z_0+Re^{i\varphi})|d\varphi & \le 2(\pi-\al)\log\left(8(A+\kappa)\right)  + \int_{|\ff| \ge \al} \frac{1}{p+R\cos\ff-\theta} d\ff
\notag \\& < 2(\pi-\al)\log\left(8(A+\kappa)\right) - 2 (\pi-\al)\log R  + 5.511,
\end{align}
referring to the already accomplished calculation in \eqref{intlogsigmamtheta}.

Adding the estimate \eqref{logzetaintegralinhalfplane} furnishes
\begin{align*}
\frac{1}{2\pi}\int_{-\pi}^{\pi} \log|\zeta(z_0+Re^{i\varphi})|d\varphi
& \le \log(A+\kappa) + \left(1-\frac{\alpha}{\pi}\right) \log 8 + \frac{\alpha(2-\theta)}{\pi}
\\&\quad + \frac{2 R \alpha \sin\al}{\pi} - \log R + \frac{\alpha}{\pi} \log\frac{1}{(w+1)}+1.755=:J^*(a)
\end{align*}
As a result, in this case we find $N\le \frac{p-a}{q-a} \left(J^*(\theta) + \log(A+\kappa)+\log \frac{p}{p-1}\right)$, and so after passing to the limit $a \searrow \theta$ -- entailing $R\to 2(1-\theta), w\to 1/2, \alpha\to 2\pi/3$-- we get
\begin{align}\label{middlecase}
N & \le \frac{3(1-\theta)}{b-\theta} \left\{J^*(\theta) + \log(A+\kappa)+\log \frac{p}{p-1}\right\} \notag
\\ &\le \frac{3(1-\theta)}{b-\theta} \bigg\{ 2 \log(A+\kappa) + 2 \log\frac{1}{1-\theta} + +C^*(\theta)\bigg\}
\end{align}
with
\begin{align*}
C^*(\theta) & = \log(2-\theta) + \frac13\log 8 + \frac23(2-\theta)+2(1-\theta)\frac{\sqrt{3}}{3}-\log 2 - \frac23 \log\frac32 +1.755
\\ &\le \log 2  +\frac43 + \frac{2\sqrt{3}}{3} - \frac23 \log\frac32 +1.755<4.665<\frac{14}{3}.
\end{align*}

\underline{Case 2.3: $0\le T \le 2$.} If all the points of $\Cm$ satisfy also $|t|\ge 1/8$, then all the estimates derived from \eqref{middlecase} with parameter value $\tau=7$ in the above case apply, because we also have that in this case if $s=\si+it \in\Cm$, then $|t|\le 2+2\diam \Cm \le 2+2\sqrt{3}<5.5<7$.

Denote now $S:=\{s=\si+it~:~ \si < \frac{1+\theta}{2} \}$ and $S^*:=\{s=\si+it~:~ \frac{1+\theta}{2} \le \si \le 1\}$. In fact, for points in $S \cap \Cm$ \eqref{zestitnotlarge} provides the same estimates than \eqref{zestitsmall} furnished in \eqref{middlecase}, for then $\max \left( \frac{1}{\si-\theta},\frac{1}{1-\si}\right)=\frac{1}{\si-\theta}$. Therefore, we need not bother with subcases whether points in $S \cap \Cm$ are close to the real axis or not.

There is need for some extra calculus only if there are points of $\Cm\cap S^*$ in the $1/8$ neighborhood of the real axis $\RR$. Here instead of \eqref{zestitsmall} we are to use \eqref{zestitnotlarge}, but now with some lower values (depending on subcases and even subarcs of $\Cm$) of the parameter $\tau$.


So assume that there is a nonempty subset $\II$--and then consisting either one or two subarcs of $\Cm$--where $\frac{1+\theta}{2} \le \si \le 1$ and $|t|\le 1/8$. Denote $\beta:=\arccos \left(-\frac{p-\frac{1+\theta}{2}}{R}\right)=\pi-\arccos\left(\frac{3(1-\theta}{2R}\right)$. The arcs with $|\ff| \in [\alpha,\beta]$ are exactly the arcs of $\Cm$ which fall into the strip $S^*$; the part satisfying $|t|\le 1/8$ make up the points of $\II$.

Consider first the subcase when $\II$ consists of two arcs, i.e. the upper and lower parts of $\Cm \cap S^*$ are both involved. Then also the arc in between (the part of $\Cm$ lying in $S=\{s~:~\si\le \frac{1+\theta}{2}\}$) lies in $|t|\le 1/8$; moreover, even the rest of $\Cm$ satisfies
$$
|t|\le \frac18+R(\sin\alpha-\sin\beta) \le \frac18+2(1-\theta)\left( \sqrt{1-\left(\frac{1-\theta}{2R}\right)^2} - \sqrt{1-\left(\frac{3(1-\theta}{2R}\right)^2} \right)=:H(a).
$$
As above, we will have $a\to \theta$ and $R\to 2(1-\theta)$, whence for close enough $a$ we will have $H(a) \le 1/8+2(1-\theta)\left( \frac{\sqrt{3}}{2} - \frac{\sqrt{7}}{4} \right)+\ve \le 0.54$, and all over the arc $\Cm$ we can calculate with $\tau=0.54$ and constant $1.54$. So in this subcase for points in $\Cm$ we have
\begin{align*}
\left|\log \zeta(p+Re^{i\ff})\right| & \le \log(A+\kappa)\qquad 
\\ & +
\begin{cases}
\log \frac98 + \log\frac{1}{\si-\theta} \quad& \textrm{if}\quad \si \le \frac{1+\theta}{2} \quad \left(\textrm{and whence} ~ |t|\le 1/8 \right)   \\
\log \frac98 +  \log\frac{1}{1-\si} \quad& \textrm{if}\quad \si \ge \frac{1+\theta}{2} \quad \textrm{and} ~ |t|\le 1/8  ~\left( \textrm{i.e.} ~ s\in  \II \right)
\\
\log1.54 + \log\frac{1}{\si-\theta} \quad& \textrm{if}\quad \si \ge \frac{1+\theta}{2} \quad \textrm{and} ~ |t|>1/8  ~\left( \textrm{i.e.} ~ s \in  (S\cap\Cm)\setminus \II \right)
\end{cases}
\end{align*}
It follows that then, compared to the above estimations, we have a gain (decrease) $\log8-\log\frac98$ on $|\ff|\ge \beta$ and at least a gain $\log8-\log1.54$ on $\alpha\le |\ff| \le \beta$, and a loss (increase) of $\log\frac{1}{1-\si} - \log\frac{1}{\si-\theta} < \log\frac{1-\theta}{1-\si}$ on $\II$.
Let us now estimate the possible loss on $\II$. On one arc we can write
\begin{align*}
\int_\alpha^\beta \log\frac{1-\theta}{1-\si} d\ff &= \int_\alpha^\beta \log\frac{1-\theta}{1-p-R\cos\ff} d\ff =-(\beta-\alpha)\log \frac{R}{1-\theta} +
\int_\alpha^\beta  \log\frac{1}{-w-\cos\ff}d\ff
\\ & =-(\beta-\alpha)\log \frac{R}{1-\theta} + \int_{\pi-\alpha}^{\pi-\beta} \log\frac{1}{-w+\cos\ff}d\ff
\\&=-(\beta-\alpha)\log \frac{R}{1-\theta}+\int_w^{\frac{3(1-\theta)}{2R}} \log\frac{1}{-w+u}\frac{du}{\sqrt{1-u^2}}
\\&\le -(\beta-\alpha)\log \frac{R}{1-\theta}+\frac{4}{\sqrt{7}} \int_w^{\frac{3(1-\theta)}{2R}} \log\frac{1}{-w+u}du
\\&= -(\beta-\alpha)\log \frac{R}{1-\theta}+\frac{4}{\sqrt{7}} \left[v-v\log v \right]_0^{\frac{3(1-\theta)}{2R}-w}
\\& = -(\beta-\alpha)\log \frac{R}{1-\theta}+ \frac{4}{\sqrt{7}} \left[ \frac{1-\theta}{2R}-\frac{1-\theta}{2R} \log\frac{1-\theta}{2R}\right]=:L(a).
\end{align*}
As a result, the difference in the estimation of $\int_{|\ff|\ge\alpha} \log|\zeta(z_0+Re^{i\varphi})|d\varphi$ as compared to \eqref{intoflogzetainstrip} in the previous case is
$$
D(a)\le 2(\pi-\beta) \log\frac98 +  2(\beta-\alpha)\log 1.54 - 2(\pi-\al)\log 8 + 2L(a)=:D_2(a).
$$
In the final estimation of $N$ the parameter $a$ was to go to $\theta$, so that it suffices to compute $D_2(\theta)=\lim_{a \to \theta+} D(a)$. Note that $a\to\theta+$ entails $w\to 1/2$, $R\to 2(1-\theta)$, $\alpha\to2\pi/3$, $\beta\to \arccos(-3/4)$ and thus in particular
$$
L(a) \to L(\theta)=-(\arccos(-3/4)-2\pi/3) \log2 +\frac{4}{\sqrt{7}} \left[ \frac{1}{4}-\frac{1}{4} \log\frac{1}{4}\right]\approx 0.677\ldots<0.7.
$$
Therefore,
$$
\frac12 D_2(\theta) <(\pi-\arccos(-3/4))\log\frac9{64} + (\arccos(-3/4)-2\pi/3)\log \frac{1.54}{8} +0.7<0.7-\frac{\pi}{3}\log 5<0,
$$
whence the total estimate of the number of zeroes $N$ can only improve compared to the previous case.

It remains to deal with the case when exactly one subarc of $S^*\cap \Cm$ meshes into the $1/8$ neighborhood of the real axis $\RR$. In this case the loss on the respective arc is at most $L(a)$ (as opposed to $2L(a)$ above), while for the rest of points we still have $|t|\le 1/8+\textrm{diam} \Cm =1/8+\sqrt{3} <2$. It follows that we can then use \eqref{zestitnotlarge} with the parameter $\tau=2$, $\tau+1=3$, and compute the difference of gains and possible losses
as
$$
D_1(a) \le (2\pi-\alpha-\beta) \log\frac{3}{8} + L(a)< (2\pi-\alpha-\beta) \log\frac{3}{8}+0.7,
$$
which is easily seen to be negative again.

Therefore, the estimations of Case 2.2 extend to all three subcases of Case 2.3 when $\II$ meshes into $S^*\cap \Cm$ in no, in one, or in two subarcs; whence the proof is completed.
\end{proof}

\bigskip
Joining the necessary number of such rectangles we get the following.

\begin{lemma}\label{l:Jensenone} Let $\theta<b<1$ and $|t|\ge 6.3$, say. Consider
the rectangle $H:=\{ z\in\CC~:~ \Re z\in [b,1],~\Im z\in
(t-1,t+1)\}$.

Then the number of zeta-zeros $n(H)$ in the
rectangle $H$ satisfy
\begin{equation}\label{zeroshightt}
n(H) \leq
\frac{\sqrt{1-\theta}}{\sqrt{(b-\theta)^3}}\left(2\log|t| + 12\log(A+\kappa) + 12\log\frac1{1-\theta} +25\right)
\end{equation}
\end{lemma}

\begin{proof} We consider the union of rectangles $H_j:= \{ z\in\CC~:~ \Re z\in [b,1],~\Im z\in
(t_j-h,t_j+h)\}$ with $h:=\frac{\sqrt{7}}{3} \sqrt{(b-\theta)(1-\theta)}$ as in the above Lemma. If $t_j=t+1-(2j-1)h\ge t-1\ge 5.222$, then the union $\cup_{j=1}^m H_j$, where
$$
m:=\left\lceil \frac{2}{h}\right\rceil = \left\lceil \frac{2}{2\frac{\sqrt{7}}{3} \sqrt{(b-\theta)(1-\theta)}} \right\rceil=\left\lceil \frac{3/\sqrt{7}}{\sqrt{(b-\theta)(1-\theta)}} \right\rceil \le \frac{\lceil 3/\sqrt{7}\rceil}{3/\sqrt{7}}~\frac{3/\sqrt{7}}{\sqrt{(b-\theta)(1-\theta)}},
$$
fully covers $H$, whence the total number of zeroes exceeds $n(H)$.

Now $\frac{3/\sqrt{7}}{\sqrt{(b-\theta)(1-\theta)}} \ge 3/\sqrt{7}$, thus taking into account $\lceil x\rceil \le \frac{\lceil 3/\sqrt{7}\rceil}{3/\sqrt{7}} x = 2\sqrt{7}/3 x$ for all $x\ge 3/\sqrt{7}$, we obtain $m\le 2/\sqrt{(b-\theta)(1-\theta)}.$

In the small rectangles $H_j$ there are at most $n_j:=\frac{1-\theta}{b-\theta} \left(\log t_j+6\log(A+\kappa)+6\log\frac{1}{1-\theta}+12.5 \right)$ zeroes, because  $t_j=t+1-(2j-1)h\ge t-1\ge 5.222$ and the above Lemma applies. So we are to estimate $N:=\sum_{j=1}^m \frac{1-\theta}{b-\theta} \left(\log t_j+6\log(A+\kappa)+6\log\frac{1}{1-\theta}+12.5 \right)$. Pairing the opposite terms thus provides
$$
2N \le \frac{1-\theta}{b-\theta}  \left\{ \sum_{j=1}^m \left(\log t_j+\log t_{m-j}\right) +2m \left(6\log(A+\kappa)+6\log\frac{1}{1-\theta}+12.5 \right)\right\}.
$$
It remains to check that $t_j+t_{m-j}\le 2t$, always, whence by the concavity of $\log t$ it follows that $\log t_j+\log t_{m-j} \le 2\log t$. Using this and the above estimate $m\le 2/\sqrt{(b-\theta)(1-\theta)}$ finally furnishes the assertion.
\end{proof}

\begin{remark} Similarly, the number of zeroes in $[b,1]\times [-iT,iT]$ can be estimated for all $T\ge 1$, say, as
\begin{align}\label{NTest}
N(b,T):&=\#\{ \rho=\beta+i\gamma~:~ \ze(\rho)=0, \beta\geq b,
|\gamma|\leq T \} \notag
\\ &\leq \frac{\sqrt{1-\theta}}{\sqrt{(b-\theta)^3}}  \left\{2 T \log_{+} T + \left(12\log(A+\kappa) + 12\log\frac1{1-\theta} +28\right)T \right\}.
\end{align}
\end{remark}
This can slightly be improved, in particular for the case when $b$
is closer to $\theta$ than to $1$, as follows.
\begin{lemma}\label{l:Littlewood} Let $\theta<b<1$ and consider
any height $T\geq 5
$ together with the rectangle $Q:=Q(b,T):=\{
z\in\CC~:~ \Re z\in [b,1],~\Im z\in (-T,T)\}$. Then the number of
zeta-zeros $N(b,T)$ in the rectangle $Q$ satisfy
\begin{equation}\label{zerosinth-corr}
N(b,T)\le \frac{1}{b-\theta}
\left\{\frac{1}{2} T \log T + \left(2 \log(A+\kappa) + \log\frac{1}{b-\theta} + 3 \right)T\right\}.
\end{equation}
\end{lemma}
\begin{proof} Instead of Jensen's theorem, we now apply
Littlewood's theorem (see \cite[p. 166]{Dav}) to the rectangle
$Q(a,T):=[a,2]\times [-iT,iT]$, with $\theta < a\leq (1+\theta)/2$
and also $a<b$, assuming momentarily that on the boundary
$\partial Q(a,T)$ there is no zero of $\ze$ (and then get it even
for other values of $a$ and $T$ by taking limits). This theorem
provides the average formula
\begin{align}\label{Littlewoodintegral}
2\pi \int_a^2 N(q,T) dq & = \int_{-T}^{T} \left[\log|\ze(a+it)| -
\log|\ze(2+it)| \right] dt
\notag \\ &\qquad + \int_{a}^{2} \left[\arg
\ze(\sigma+iT) - \arg \ze(\sigma-iT) \right] d\sigma,
\end{align}
where $\arg \ze$ is defined by continuous variation from $\ze(2+it)$.

As in \cite{Dav}, the estimation of the second integral is executed by an appeal to Backlund's observation: $\arg \zs = \Im \log \ze(s)= \frac{1}{2i} \left[ \log \zs
-\overline{\log \zs} \right]$, hence the second integral is
$$
\int_{a}^{2} \left[\arg \ze(\sigma+iT) - \arg \ze(\sigma-iT)
\right] d\sigma = -i \int_{a}^{2} \left[\log \ze(\sigma+iT) - \log
\ze(\sigma-iT) \right] d\sigma
$$
At $\sigma=2$, Lemma \ref{l:oneperzeta}, \eqref{reciprok} provides
that either $|\Re \ze(2\pm iT)| \geq 1/(2\sqrt{2}(A+\kappa))$, or
$|\Im \ze(2\pm iT)| \geq 1/(2\sqrt{2}(A+\kappa))$; e.g. in the
first case let
$$
g(s):=g_T(s):= \frac12\left[ \ze(s+iT)+\ze(s-iT)\right].
$$
The number of zeroes $\nu(\sigma)$ of $g$ along the segment
$[\sigma,2]$ is the number of cases when $\Re \ze(\al\pm iT)$
vanishes along this segment, hence by the definition using
continuous variation, we must have $|\arg \ze(\sigma\pm iT)| \leq
(\nu(\sigma)+1)\pi$. Therefore,
$$
\frac{1}{\pi}\int_a^2 \arg\ze(\sigma\pm iT) d\sigma \leq \int_a^2
\nu(\sigma)d\sigma + (2-a) \leq \int_a^2 n_g(2-\sigma) d\sigma +2,
$$
where now $n(2-\sigma):=n_g(2-\sigma):=\#\{s~:~|s-2|\leq
r:=2-\sigma,~g(s)=0\}$, which is clearly at least as large as
$\nu(\sigma)$. Now we use Jensen's formula to estimate the last
integral as
$$
\int_a^2 n_g(2-\sigma) d\sigma =\int_0^{2-a} n(r)dr < 2
\int_0^{2-a} \frac{n(r)}{r} dr \leq \frac{1}{\pi} \int_0^{2\pi}
\log|g(2+(2-a)e^{i\varphi})| d\varphi - 2 \log|g(2)|.
$$
By the above, $- \log|g(2)| = -\log |\Re\ze(2+ iT) | \leq
\log(2\sqrt{2}(A+\kappa))$, while for the integral we can apply $|
g(2+(2-a)e^{i\varphi})| \leq \max \left(
|\ze(2+iT+(2-a)e^{i\varphi})|,~|\ze(2-iT+(2-a)e^{i\varphi})|\right)$.
This last expression can be estimated using the uniform estimates
of Lemma \ref{l:zkiss}, \eqref{zsgenlarget} (note that $T\geq 5$
so the variables considered are all in the range $|t|\geq \si$). We
are led to
\begin{align*}
\int_a^2 n_g(2-\sigma) d\sigma &  \leq 2\left(
\log\frac{\sqrt{2}(A+\kappa)}{a-\theta} + \log(T+2) +
\log(2\sqrt{2}(A+\kappa))\right) \\
& \leq 2 \log T + 4 \log(A+\kappa) + 2 \log\frac{1}{a-\theta} +
3.4.
\end{align*}
In all for the argumentum integrals we obtain
\begin{align}\label{argint}
\int_{a}^{2} \big[\arg \ze(\sigma+iT) & - \arg \ze(\sigma-iT)
\big] d\sigma \leq 2\pi\left( \int_a^2 n_g(2-\sigma) d\sigma +2\right) \notag
\\ & \leq 4 \pi \log T + 8 \pi \log(A+\kappa) + 4 \pi \log\frac{1}{a-\theta} + 2\pi\cdot 5.4.
\end{align}
Now let us go to the evaluation of the other integral. Using once
again the estimate of Lemma \ref{l:oneperzeta}, \eqref{reciprok},
the term along the $\sigma=2$ line is of no problem, contributing
\begin{equation}\label{logzetaonthe2line}
\int_{-T}^{T}  - \log|\ze(2+it)|  dt \leq 2T  \log \left(2(A+\kappa) \right).
\end{equation}
For the most essential part, the integral along the $\sigma=a$
line, in the finite range between $[-4,4]$ an application of
\eqref{zsgeneral} of Lemma \ref{l:zkiss} yields, using also
$\theta <a\leq  (1+\theta)/2$,
\begin{align*}
\int_0^4 \log |\ze (a+it)| dt & \leq \int_0^4 \log \left(
\frac{A|a+it|}{a-\theta} + \frac{\kappa}{|(1-(a+it)|}\right) dt
\leq \int_0^4 \log \left( \frac{(A+\kappa)\sqrt{1+t^2}}{a-\theta}
\right) dt
\\ & = 4 \log \left(\frac{A+\kappa}{a-\theta} \right) + \frac 12
\int_0^4 \log (1+t^2) dt = 4 \log \left(\frac{A+\kappa}{a-\theta}
\right) +  2.992244...
\end{align*}
For the range with $4\leq|t|\leq T$ we can refer to the estimate \eqref{zestitsmall} of Lemma \ref{r:zestismallt}. 
From that and the above we obtain
\begin{align*}
\int_{-T}^{T} \log|\ze(a+it)| dt & < 8 \log
\left(\frac{A+\kappa}{a-\theta} \right) +  6 +  2\int_4^T \left(
\log \frac{(A+\kappa)}{a-\theta} + \log (t+1) \right) dt
\\ & \leq 2T \log
\left(\frac{A+\kappa}{a-\theta} \right) + 6+
2\left\{((T+1) \log (T+1) -(T+1))  -( 5\log 5-5) \right\} \\ & = 2T\log T + 2T \log\frac{T+1}{T} +2\log (T+1) + \left(\log \frac{A+\kappa}{a-\theta} -1\right) \cdot 2 T + 14 - 10 \log 5
\\ & < 2T\log T + \left(\log \frac{A+\kappa}{a-\theta} -1\right) \cdot 2 T + 2 \log(T+1) + 16-10\log 5.
\end{align*}
Collecting \eqref{Littlewoodintegral}, \eqref{argint}, \eqref{logzetaonthe2line} and the last estimates we obtain from Littlewood's theorem that
\begin{align*}
2\pi \int_a^2 N(q,T) dq & \leq 2T \log T +
\left(\log\frac{2(A+\kappa)^2}{a-\theta} -1 \right) 2T + 4\pi
\log T +2\log(T+1) \\ &\qquad\qquad\qquad + 8 \pi \log(A+\kappa) + 4\pi
\log\frac{1}{a-\theta} + 2\pi\cdot5.4+ 16-10\log 5.
\end{align*}

Here we choose $a:=(b+3\theta)/4=\theta+\frac14(b-\theta)$ and apply $(b-a) N(b,T) \leq \int_a^2 N(q,T) dq$ to obtain (taking into account $b-a=\frac34 (b-\theta)$, $a-\theta=\frac14 (b-\theta)$ and also $T\ge 5$)
\begin{align*}
N(b,T) & \leq \frac{1}{b-a} \bigg\{ \frac{1}{\pi} T \log T +
\left(\frac{2}{\pi} \log(A+\kappa) + \frac{3\log 2-1}{\pi} +
\frac1{\pi} \log \frac{1}{b-\theta} \right) T
\\ &  \qquad \qquad\qquad+ 2 \log T + \frac{1}{\pi} \log(T+1) + 4
\log(A+\kappa) + 2 \log\frac{4}{b-\theta} + 5.4 + \frac{8-10\log 5}{\pi}\bigg\} \\
& \leq \frac{1}{\frac{3}{4}(b-\theta)} \bigg\{\frac{1}{\pi} T \log T +
\left(\left(\frac{2}{\pi}+\frac{4}{T}\right) \log(A+\kappa) + \left( \frac{1}{\pi}+\frac{2}{T}\right) \log\frac{1}{b-\theta} + 0.35 \right) T
\\ &  \qquad \qquad\qquad  + \left(2 \frac{\log T}{T} + \frac{1}{\pi} \frac{\log(T+1)}{T} + \left(2 \log 4 + 5.4 + \frac{8-10\log 5}{\pi}\right)\frac{1}{T} \right)  T \bigg\}
\\ & \leq \frac{1}{\frac{3}{4}(b-\theta)} \bigg\{\frac{1}{\pi} T \log T + \bigg[
\left(\frac{2}{\pi} +\frac45\right) \log(A+\kappa) + \left(\frac{1}{\pi}+\frac{2}{5} \right)\log\frac{1}{b-\theta}
\\ & \qquad \qquad\qquad\qquad \qquad+0.35 + 2 \frac{\log 5}{5} + \frac{1}{\pi} \frac{\log 6}{5} + \left(2 \log 4 + 5.4 + \frac{8-10\log 5}{\pi}\right)\frac{1}{5} \bigg]  T \bigg\}
\\ & \leq \frac{4}{3(b-\theta)} \left\{\frac{1}{\pi} T \log T + \left[
1.44 \log(A+\kappa) + 0.72 \log\frac{1}{b-\theta} +2.23\right]T\right\}
\\ & \le \frac{1}{b-\theta} \left\{0.425 ~T \log T + \left[1.92 \log(A+\kappa) + 0.96 \log\frac{1}{b-\theta} + 2.98\right]T\right\}
\end{align*}
and the assertion follows.
\end{proof}

\begin{lemma}\label{c:zerosinrange}
Let $\theta<b<1$ and consider any heights $T>R\geq 5 $ together
with the rectangle $Q:=Q(b,R,T):=\{ z\in\CC~:~ \Re z\in [b,1],~\Im
z\in (R,T)\}$.

Then the number of zeta-zeros $N(b,R,T)$ in the rectangle $Q$ satisfies
\begin{equation}\label{zerosbetween}
N(b,R,T) \leq\frac{1}{b-\theta} \left\{ \frac{4}{3\pi} (T-R) \left(\log T + \log\left(\frac{11.4 (A+\kappa)^2}{b-\theta}\right)\right)  + \frac{16}{3}  \log\left(\frac{60 (A+\kappa)^2}{b-\theta}\right)\right\}.
\end{equation}

In particular, for the zeroes between $T-1$ and $T+1$ we have for
$T\geq 6$
\begin{align}\label{zerosbetweenone}
N(b,T-1,T+1) \leq \frac{1}{(b-\theta)} \left\{ 0.85 \log T +
6.2 \log\left( \frac{(A+\kappa)^2}{b-\theta}\right) + 24 \right\}.
\end{align}
\end{lemma}

\begin{proof} We can apply the Littlewood theorem not only to $Q(b,T)$,
but also to the rectangle $Q(b,R):=[b,2]\times [-iR,iR]$:
moreover, we can simply subtract the contributions of the two
formulae to estimate the contribution of the zeroes in between.
Thus we obtain (again with $a:=(b+3\theta)/4$) the estimations
\begin{align*}
2\pi (b-a)[N(b,T)-N(b,R)] &\leq 2\pi \int_a^2
[N(q,T)-N(q,R)]dq \\& = \int_{[-T,-R]\cup[R,T]}
\left[\log|\ze(a+it)| - \log|\ze(2+it)| \right] dt
\\ &\qquad \qquad  + \int_{a}^{2} \left[\arg \ze(\sigma+iT) - \arg
\ze(\sigma-iT) \right] d\sigma
\\ &\qquad\qquad  - \int_{a}^{2} \left[\arg
\ze(\sigma+iR) - \arg \ze(\sigma-iR) \right] d\sigma.
\end{align*}
As before in \eqref{argint}, for the difference of the two argumentum integrals we can infer -- also writing in the trivial equality $\log\frac{1}{a-\theta} = \log \frac{1}{b-\theta} + \log 4$-- the estimates
\begin{align*}
\int_{a}^{2} (\dots T\dots )d\sigma & - \int_{a}^{2} (\dots R \dots
) d\sigma
\\& \leq 8 \pi \log T + 16\pi \log(A+\kappa) + 8\pi
\log\frac{1}{b-\theta} +8\pi \log 4  + 2\pi 10.8.
\end{align*}
On the other hand for the contribution of the main term integrals
of $\log|\ze|$ we have a better estimation, due to the fact that
here the part over $[-R,R]$ of the integral is canceled. For the
part over $\sigma=2$ by the same uniform estimation using Lemma \ref{l:oneperzeta}, \eqref{reciprok}, we obtain
$$
\int_{[-T,-R]\cup[R,T]} - \log|\ze(2+it)|  dt \leq 2(T-R) \log
\left( 2(A+\kappa) \right).
$$
Lastly, the contribution over the $\sigma=a$ line can be estimated using \eqref{zestitsmall} of Lemma \ref{r:zestismallt} as
\begin{align*}
\int_{[-T,-R]\cup[R,T]} \log|\ze(a+it)| dt &\leq \int_{[-T,-R]\cup[R,T]} \left( \log
\frac{\sqrt{2}(A+\kappa)}{a-\theta} + \log |t| \right) dt.
\\ & \leq 2  (T-R) \left(\log T+ \log \frac{4\sqrt{2}(A+\kappa)}{b-\theta}\right)
\end{align*}
So collecting our estimates and writing in $b-a=\frac34 (b-\theta)$ we are led to
\begin{align*}
2\pi \frac34 (b-\theta) N(b,R,T) & \leq 2\pi \int_a^2 [N(q,T)-N(q,R)]dq
\\ & \leq 2 (T-R) \left( \log T + \log
\left(\frac{8\sqrt{2}(A+\kappa)^2}{b-\theta}\right) \right)
\\ & + 8 \pi \log T  + 16 \pi \log(A+\kappa) +
8\pi  \log\frac{1}{b-\theta} +2\pi 10.8+8\pi\log 4
\\&\leq 2 (T-R)\left(\log T + \log\left(\frac{11.4 (A+\kappa)^2}{b-\theta}\right)\right)  + 8 \pi  \left(\log\left(\frac{59.52 (A+\kappa)^2}{b-\theta}\right)\right).
\end{align*}
The inequality \eqref{zerosbetween} follows.

This formula applies also for $N(b,T+1,T-1)$ on noting that
$\log(T+1)\leq \log T + \log\frac76$ so that
\begin{align*}
N(b,T-1,T+1) & \leq \frac{1}{b-\theta} \left\{ \frac{8}{3\pi}\left(\log T + \log\left(\frac{\frac76 11.4 (A+\kappa)^2}{b-\theta}\right)\right)  + \frac{16}{3}  \left(\log\left(\frac{59.52 (A+\kappa)^2}{b-\theta}\right)\right) \right\}
\\& < \frac{1}{b-\theta} \left\{ 0.8489 \log T + 6.2 \log\left(\frac{(A+\kappa)^2}{b-\theta}\right)  + \left(0.8489 \log 13.3 + \frac{16}{3}  \log 59.52 \right) \right\}.
\end{align*}
From here a little numerical calculus ends the proof of the Lemma.
\end{proof}

\section{The logarithmic derivative of the Beurling $\zeta$}
\label{sec:logder}

\begin{lemma}\label{l:borcar} Let $z=a+it_0$ with $|t_0| \geq e^{5/4}+\sqrt{3}=5.222\ldots$ and $\theta<a\leq 1$. With $\delta:=(a-\theta)/3$ denote by $S$ the
(multi)set of the $\ze$-zeroes (listed according to multiplicity)
not farther from $z$ than $\delta$. Then we have
\begin{align}\label{zlogprime}
\left|\frac{\ze'}{\ze}(z)-\sum_{\rho\in S} \frac{1}{z-\rho}
\right| & < \frac{9(1-\theta)}{(a-\theta)^2}
\left(22.5+14\log(A+\kappa)+14\log \frac{1}{a-\theta} + 5\log |t_0|\right).
\end{align}

Furthermore, for $0 \le |t_0| \le 5.23$ an analogous estimate (without any term containing $\log |t_0|$) holds true:
\begin{equation}\label{zlogprime-tsmall}
\left|\frac{\ze'}{\ze}(z)+\frac{1}{z-1}-\sum_{\rho\in S} \frac{1}{z-\rho}
\right|  \le \frac{9(1-\theta)}{(a-\theta)^2}
\left(34+14\log(A+\kappa)+18\log \frac{1}{a-\theta}\right)
\end{equation}

\end{lemma}
\begin{proof} Let now $z_0:=p+it_0$, $p:=1+(1-\theta)$, and denote
by $D_j$ the disk around it with the radii
$R_j:=p-(\theta+j\delta)$ for $j=1,2,3$, where we have chosen
$\delta:=(a-\theta)/3$. Observe that then $z\in\partial D_3\subset D_3$.

By Lemma \ref{l:Jensen} and Remark \ref{r:alsoindisk} (applied with $q:=\theta+\delta=\theta+\frac23 (b-\theta)$ i.e. with $b:=\theta+\frac32\delta$) we
know that the number $N$ of $\ze$-zeroes in the disk $D_1$ satisfies whenever $t_0 \ge 5.222\ldots$ the estimate
\begin{align}\label{zerosinDone}
N & \leq \frac{1-\theta}{\frac{3}{2}\de} \left(12.5+6\log(A+\kappa)+6\log \frac{1}{1-\theta} + \log t_0 \right)
\\ \notag & =\frac{1-\theta}{\de} \left(\frac{25}{3}+4\log(A+\kappa)+4\log \frac{1}{1-\theta} + \frac23 \log t_0\right).
\end{align}
Now denote the (multi)set of all these zeroes (again taking them
according to multiplicities) as $S'$ and define
\begin{equation}\label{eq:Blaschkefactor}
P(s):=\prod_{\rho\in S'}\al_{\rho}(s),\qquad \textrm{where} \qquad
\al(s):=\al_{\rho}(s):= \frac{(s-z_0)R_1-(\rho-z_0)R_1}
{R_1^2-(s-z_0)\overline{(\rho-z_0)}}.
\end{equation}
Let us consider now
$$
f(s):=\frac{\zeta(s)}{P(s)},
$$
which is a regular function in $D_1$, moreover, by construction,
it is non-vanishing in $D_1$. Hence also $g(s):=\log f(s)$ is
regular at least in $D_1$; by appropriate choice of the logarithm
we may assume that $\Im g(z_0)= \textrm{arg}~ f(z_0) \in
[-\pi,\pi]$.

First we estimate the order of magnitude of $f$ on the perimeter
of $D_1$. If $s=\si+it\in\partial D_1$, then any factor of $P(s)$ is
exactly 1 in modulus, hence $|f(s)|= |\zs|$.

Note that $\sigma-\theta$ is at least $\delta$, $\sigma<p+R_1<2p-\theta < 4$, $R_1<p-\theta<2$ and $t>t_0-R_1>3.22$. Moreover, a little coordinate geometry reveals that the distance of the point $z_0=p+it_0$ from the line $\si=t$ is $\frac{1}{\sqrt2} (\Im z_0 -\Re z_0)\ge \frac{1}{\sqrt2} (5.2-p)>\frac{3}{\sqrt2}>2>R_1$, whence the whole disk $D_1$ lies in the domain $\si<t$, and the estimate \eqref{zsgenlarget} of Lemma \ref{l:zkiss} applies furnishing

$$
|f(s)|= |\zs| \leq \sqrt{2} \frac{(A+\kappa) t}{\si-\theta} \qquad
(s\in\partial D_1).
$$
As $g(s):=\log f(s)$ is analytic on $D_1$, we can apply the
Borel-Caratheodory lemma:
$$
\max_{|s-z_0|\leq R_2} |g(s)-g(z_0)| \leq
\frac{2R_2}{R_1-R_2}\max_{|s-z_0|\leq R_1} \Re \left( g(s) -
g(z_0) \right).
$$
That is, with a slight reformulation we obtain for $s\in D_2$
$$
|g(s)| \leq |g(s)-g(z_0)| + |\Im g(z_0)| + |\Re g(z_0)| \leq \pi +
\frac{2R_2}{\delta}\max_{|s-z_0|\leq R_1} \Re g(s) +\left(
\frac{2R_2}{\delta} \pm 1 \right) \Re (-g(z_0)),
$$
where here by the choice of our radii and parameters, it is clear
that the sign of the coefficient in front of $\Re (-g(z_0))$ is
positive. Firstly,
$$
\Re g(s) = \log|f(s)| \leq \log\left( \sqrt{2}
\frac{(A+\kappa)|t|}{\si-\theta}\right) \le \log\left( \sqrt{2}
\frac{(A+\kappa)(t_0+R_1)}{\delta}\right) \qquad (s\in\partial D_1).
$$
that is, taking into account $R_1/t_0<2/5.2$, we get for the maximum of $\Re g(s)$ on $\partial D_1$
$$
\max_{|s-z_0|\leq R_1} \Re g(s)  \leq \log\left(2\frac{(A+\kappa)t_0}{\delta}\right).
$$
Application of the Borel-Caratheodory lemma on $D_2$ thus yields
$$
\max_{|s-z_0|\leq R_2} |g(s)| \le \pi + \frac{4(1-\theta)-4\de}{\de}
\left(\log(2(A+\kappa)) +\log t_0+ \log\frac1{\delta} \right) + \left(\frac{4(1-\theta)-4\de}{\de} \pm 1\right) \Re(- g(z_0))
.
$$
Since the factors $\al(s)$ of $P(s)$ are all less than $1$ at
$z_0$, by \eqref{reciprok} we obtain
$$
- \Re g(z_0) = \log\left| \frac{1}{\ze(z_0)}\right| + \log
\left|P(z_0)\right| \leq \log\frac{p(A+\kappa)}{p-1} \le \log\frac{2(A+\kappa)}{1-\theta}.
$$
Hence taking into account $\frac1{\de}=\frac3{a-\theta}$ and $\frac1{1-\theta}\le\frac1{a-\theta}$, we are led to
\begin{align*}
\max_{|s-z_0| \leq R_2} |g(s)|
& \le \pi+  \left(\frac{4(1-\theta)}{\de} -4\right)  \left(\log(2(A+\kappa)) +\log t_0+ \log\frac1{\delta} \right) + \left(\frac{4(1-\theta)}{\de} -3\right) \log\frac{2(A+\kappa)}{1-\theta}
\\ &
\le \frac{4(1-\theta)}{\de}
\left(\log\left(12(A+\kappa)^2\right) +\log t_0+ 2 \log\frac1{a-\theta} \right) +\pi -4\log6-3\log2
\end{align*}
Note that the sum of the constants $\pi -4\log6-3\log2$ is negative, whence can be dropped.
Finally we apply the usual Cauchy estimate to the regular function
$g=\log f$ in $D_2$ to obtain a generally valid estimate of the
derivative in $D_3$. Thus the radius of the estimation is
$\delta$, which will divide on the right hand side to give
\begin{align}\label{MgD3}
\max_{|s-z_0|\leq R_3} |g'(s)| &
\le \frac{4(1-\theta)}{\de^2}
\left(\log12 +2\log (A+\kappa)+\log t_0+ 2 \log\frac1{a-\theta} \right).
\end{align}

In view of the definition of $f(s)$, $P(s)$ and $S'$ we have
$f'/f=\ze'/\ze -P'/P$, hence
$$
g'(s)= \frac{\ze'}{\ze}(s) - \sum_{\rho\in S'} \frac{1}{s-\rho} -
\sum_{\rho\in S'} \frac{\overline{(\rho-z_0)}}
{R_1^2-(s-z_0)\overline{(\rho-z_0)}},
$$
Note that the last sum has $N$ terms, and for any $s\in D_3$ the difference in the denominator is at least $2\delta R_1$ in modulus, while the numerator is at most $R_1$. Therefore, we have
$$
\left| \sum_{\rho\in S'} \frac{\overline{(\rho-z_0)}}
{R_1^2-(s-z_0)\overline{(\rho-z_0)}}\right| \le \frac{N}{2\de}.
$$
Finally, we need to take care of the difference between $S$ and $S'$.

Here we specify to the point $z=a+it_0\in D_3$.
Clearly, $S\subset S'$, and for any $\rho\in S'\setminus S$
$|z-\rho|>\delta$ by construction (definition of $S$). Therefore, $\left|\sum_{\rho\in S'\setminus S} \frac1{z-\rho}\right|$ does not exceed $N/\delta$, so we get
$$
\left|g'(z)-\left(\frac{\ze'}{\ze}(s)-\sum_{\rho\in S} \frac{1}{z-\rho}
\right)\right| \le \frac{3N}{2\de} \le \frac{1-\theta}{\de^2} \left(12.5+6\log(A+\kappa)+6\log \frac{1}{1-\theta} + \log t_0\right).
$$
Combining with \eqref{MgD3} we obtain the assertion.

\bigskip
If $z=a+it_0$ with $|t_0| \le e^{5/4}+\sqrt{3}$, then an analogous proof should be followed with $\zeta(s)(s-1)$ replacing $\zeta(s)$ in the argument.

Again, let $N$ denote the number of $\zeta$-zeroes in the disk $D_1$. Analogously to  \eqref{zerosinDone} now the estimate \eqref{zerosinH-smallt} of Lemma \ref{l:Jensen} provides (taking into account Remark \ref{r:alsoindisk}, too) $N \leq \frac{1-\theta}{b-\theta}\left(14 +6\log(A+\kappa) + 6\log\frac1{1-\theta}\right)$, where also here the value of $b$ is chosen as $\theta+\frac32\de$, so that we get\footnote{Note that Lemma \ref{l:Littlewood} \eqref{zerosinth-corr} can only provide with $T= 5$ and with $b:=\theta+\de$ the less sharp estimate
\begin{align*}
N(b,T) & \le \frac{1}{b-\theta}
\left\{\frac{1}{2} T \log T + \left(2 \log(A+\kappa) + \log\frac{1}{b-\theta} + 3 \right)T\right\}
\\ &\le \frac{1}{\de}
\left\{\frac{5}{2} \log 5 + \left(2 \log(A+\kappa) + \log\frac{3}{a-\theta} + 3 \right)5\right\}\le \frac{1}{\de}
\left\{25 + 10 \log(A+\kappa) + 5 \log\frac{1}{a-\theta}\right\}
\end{align*}
}
\begin{equation}\label{eq:nwhentsmall}
N\le \frac{1}{\de} \left(\frac{28}3 +4\log(A+\kappa) + 4\log\frac1{1-\theta}\right).
\end{equation}
Let us define now similarly as above the function $F(s):=\zs (s-1)/P(s)$, where $P(s)$ is the Blaschke-type product of zero factors as in \eqref{eq:Blaschkefactor}.

The function $\zs (s-1)$ has an estimate, uniformly valid all over $\si>\theta$. Namely, according to \eqref{zssmin1} we have $|\zs (s-1)| \le \frac{|s||s-1|A}{\si-\theta} +\kappa\le \frac{(\max(|s|,|s-1|))^2 (A+\kappa)}{\si-\theta}$. As the Blaschke factors have absolute value 1 on the circle $\partial D_1$, we therefore obtain
$$
|F(s)| \le \frac{(\max(|s|,|s-1|))^2(A+\kappa)}{\si-\theta} \qquad
(s\in\partial D_1).
$$
As $G(s):=\log F(s)$ is analytic on $D_1$, we can apply the
Borel-Caratheodory lemma:
$$
\max_{|s-z_0|\leq R_2} |G(s)-G(z_0)| \leq
\frac{2R_2}{R_1-R_2}\max_{|s-z_0|\leq R_1} \Re \left( G(s) -
G(z_0) \right).
$$
That is, with the same slight reformulation as above we obtain for $s\in D_2$
$$
|G(s)| \leq |G(s)-G(z_0)| + |\Im G(z_0)| + |\Re G(z_0)| \leq \pi +
\frac{2R_2}{\delta}\max_{|s-z_0|\leq R_1} \Re G(s) +\left(
\frac{2R_2}{\delta} \pm 1 \right) \Re (-G(z_0)),
$$
We are to estimate the maximum of $\Re G(s)=\log |F(s)|$ -- or, equivalently, of $|F(s)|$ -- on $\partial D_1$. For $\si\ge 1$ we already have $\si-\theta\ge 1-\theta\ge a-\theta$, and we can easily get $\Re G(s) \le \log \frac{1}{a-\theta} + \log(|s|^2(A+\kappa)) \le \log \frac{1}{a-\theta}+\log ((7.23^2+4^2)(A+\kappa)) < \log \frac{1}{a-\theta}+\log (69(A+\kappa))$.

Otherwise write $\si+it=z_0+R_1e^{i\ff}$ where $\ff:=\arg (s-z_0) \in [-\pi,-\alpha] \cup [\al,\pi]$. Then $\si-\theta=p-\theta+R_1\cos \ff=\de+R_1(1+\cos\ff)$ and $t=t_0+R_1\sin\ff$, so that in particular it is clear that we can restrict to the case of positive $\sin \ff$, i.e. when $\al\le \ff \le \pi$. Introducing the new variables $\psi:=\frac{\pi-\ff}{2} \in [0,\pi/2]$ and $x:=2R_1\sin\psi \in [0,4]$, we can write
\begin{align*}
\frac{|F(s)|}{A+\kappa} & \le \frac{\si^2+t^2}{\de+R_1(1+\cos\ff)} \le \frac{1+(t_0+R_1 2\sin\psi\cos\psi)^2}{\de+2R_1\sin^2\psi}
\\ & < 2R_1 \frac{28.4+4R_1 \sin\psi + 4R_1^2 \sin^2\psi}{2R_1\de+4R_1^2\sin^2\psi}=2R_1 \frac{28.4+2x + x^2}{2R_1\de+x^2}
\end{align*}
If $q(x):=\frac{A+x+x^2}{B+x^2}$ with $A:=28.4$ and $B=2R_1\de \le 4\frac{a-\theta}{3} <4/3$,  then for $x\in [0,4]$ the maximum of $q(x)$ has to occur somewhere in $[0,\frac{B}{A-B}]\subset [0,B/27]$, for $x>\frac{B}{A-B}$ the derivative $q'(x)$ is easily seen\footnote{Its sign is equivalent to that of $B+(B-A)x-x^2$.} to be negative. However, for $0\le x \le B/27$ we have $q(x) \le \frac1{B} (A+B/27+(B/27)^2) < \frac1{B} (28.4+0.4+0.1)<29/B$, whence we find for all $s$ with $\si<1$ the estimate $\frac{|F(s)|}{A+\kappa} \le 29/\de$ and $\log|F(s)| \le \log 67 + \log(A+\kappa) + \log\frac{1}{a-\theta}$.

Combining with the above we thus obtain for all cases the estimate
$$
\Re G(s) \le \log (69(A+\kappa)) + \frac{1}{a-\theta}  \qquad ( s \in \partial D_1).
$$

The value of $-\Re G(z_0)=\log\left|\frac{1}{\ze(z_0)(z_0-1)}\right|+\log|P(z_0)|$ can be estimated similarly as above with reference to \eqref{reciprok}, $|z_0-1|\ge \Re (z_0-1)=1-\theta$ and $P(z_0)\le 1$ furnishing
\begin{align*}
-\Re G(z_0) & \le \log\frac{1}{\ze(z_0)} +\log\frac1{1-\theta}
\le \log\left(\frac{2(A+\kappa)}{(1-\theta)^2}\right)\le \log\left(\frac{2(A+\kappa)}{(a-\theta)^2}\right).
\end{align*}
Collecting these terms we get similarly to \eqref{MgD3}
\begin{align}\label{MgD3-smallt}
\max_{|s-z_0|\leq R_3} |G'(s)| & \le \frac{1}{\de} \left\{\pi+  \left(\frac{4(1-\theta)}{\de} -4\right)  \left(\log\left(\frac{69(A+\kappa)}{a-\theta}\right)\right) + \left(\frac{4(1-\theta)}{\de} -3\right) \log\frac{2(A+\kappa)}{(a-\theta)^2} \right\}
\notag
\\ & \le \frac{4(1-\theta)}{\de^2}
\left(\log 138 +2\log (A+\kappa)+ 3 \log\frac1{a-\theta} \right) + \frac1{\de}\left\{\pi - 4\log 69 - 3 \log 2\right\} \notag
\\ & < \frac{4(1-\theta)}{\de^2}
\left(4.93 +2\log (A+\kappa)+ 3 \log\frac1{a-\theta} \right)
\end{align}
Much like the above, specifying to $s=z$ and using the estimate \eqref{eq:nwhentsmall} of $N$ now we infer
$$
\left|G'(z)-\left(\frac{\ze'}{\ze}(z)+\frac{1}{z-1}-\sum_{\rho\in S} \frac{1}{z-\rho}
\right)\right| \le \frac{3N}{2\de} \le \frac{1-\theta}{\de^2} \left(14+6\log(A+\kappa)+6\log \frac{1}{1-\theta}\right),
$$
whence combining with \eqref{MgD3-smallt} we obtain the assertion.
\end{proof}

\begin{lemma}\label{l:path-translates} For any given parameter
$\theta<b<1$, and for any finite and symmetric to zero set $\A\subset[-iB,iB]$ of cardinality $\#\A=n$, there exists a broken line $\Gamma=\Gamma_b^{\A}$, symmetric to the real axis and consisting of horizontal and vertical line segments only, so that its upper half is
$$
\Gamma_{+}= \bigcup_{k=1}^{\infty}
\{[\sigma_{k-1}+it_{k-1},\sigma_{k-1}+it_{k}] \cup
[\sigma_{k-1}+it_{k},\sigma_{k}+it_{k}]\}
$$
with $\sigma_j\in [\frac{b+\theta}{2},b]$, ($j\in\NN$),  $t_0=0$, $t_1\in[4,5]$ and $t_j\in
[t_{j-1}+1,t_{j-1}+2]$ $(j\geq 2)$ and satisfying that the
distance of any $\A$-translate $\rho+ i\alpha ~(i\alpha\in\A)$ of a $\zeta$-zero $\rho$ from any point $s=t+i\sigma \in \Gamma$ is at least $d:=d(t):=d(b,\theta,n,B;t)$ with
\begin{equation}\label{ddist-corr}
d(t):=\frac{(b-\theta)^2}{4n \left(4.4 \log(|t|+B+5) + 51 \log (A+\kappa) + 31 \log\frac{1}{b-\theta}+ 113\right)}.
\end{equation}
Moreover, the same separation from translates of $\zeta$-zeros holds also for the
whole horizontal line segments $H_k:=[\frac{b+\theta}{2}+it_k,2+it_k]$, $k=1,\dots,\infty$, and their reflections $\overline{H_k}:=[\frac{b+\theta}{2}-it_k,2-it_k]$, $k=1,\dots,\infty$, and furthermore the same separation holds from the translated singularity points $1+i\al$ of $\zeta$, too.
\end{lemma}

\begin{proof} Let us denote $a:=\frac{b+\theta}{2}$, and $p:=0.51\theta+0.49b$. Note that $a-p=0.01(b-\theta)>d$, so that any translated zeta-zero with real part $\beta:=\Re \rho=\Re(\rho+i\alpha)\le p~(\alpha\in\A)$ necessarily has distance at least $d$ from the whole curve $\Gamma$ and from all the $H_k$, all lying in the half-plane $\Re s\ge a$.

According to \eqref{zerosinth-corr} in Lemma \ref{l:Littlewood}, the number of zeroes in the rectangle
$Q(p,6):=[p,1]\times [-i6,i6]$ satisfies
\begin{align}\label{Np6}
N(p,6) &
\le \frac{1/0.49}{b-\theta} \left\{ \frac1{2}6 \log 6 + 12\log(A+\kappa) + 6 \log\frac{1/0.49}{b-\theta} + 18 \right\} \notag \\& \le \frac{1}{b-\theta}
\left\{25\log(A+\kappa)+12.5 \log\dfrac{1}{b-\theta}+57\right\}=:X.
\end{align}

Let now $R\ge 5$ be any parameter. Then for estimating the number of $\ze$-zeroes in the rectangle
$[p,1]\times[iR,i(R+1)]$ we can apply \eqref{zerosbetween} Lemma \ref{c:zerosinrange} with $T:=R+1$ to get
\begin{align}\label{Np5R}
N(p, R, R+1)  & \le
\frac{1}{p-\theta} \left\{ \frac{4}{3\pi} \left(\log (R+1) + \log\left(\frac{11.4 (A+\kappa)^2}{p-\theta}\right)\right)  + \frac{16}{3}  \log\left(\frac{60 (A+\kappa)^2}{p-\theta}\right)\right\}
\notag
\\& < \frac{1}{b-\theta} \left\{\log (R+1) +  \log\left(\frac{11.4 (A+\kappa)^2}{0.49(b-\theta)}\right)  + \frac{16}{3\cdot0.49}  \log\left(\frac{60 (A+\kappa)^2}{0.49(b-\theta)}\right)\right\} \notag
\\ \notag & < \frac{1}{b-\theta} \left\{\log (R+1) + 24 \log(A+\kappa)+12 \log\dfrac{1}{b-\theta}+56\right\}
\\ & < \frac{1}{b-\theta} \log (R+1)+X =  :X_0(R).
\end{align}
Note that $X_0(R)$ in \eqref{Np5R} is increasing in function of $R\ge 5$, and it exceeds \eqref{Np6}, always.

First we describe how we choose the horizontal line segments $H_k$, i.e. the values of the coordinates $t_k$, to satisfy that the segments $H_k:=[a+it_k,2+it_k]$ avoid $\ze$-zeroes and translated zeta-zeroes by at least $d$. We need to construct only $\{t_k\}_{k\in \NN}$ in view of the obvious symmetry (of both the curve $\Gamma$ and the set of all translated zeta-zeroes) with respect to $\RR$.

Our choice is inductive; first we choose $t_1\in [4,5]$, and then inductively $t_k\in [t_{k-1}+1,t_{k-1}+2]$.
(Recall that $t_0:=0$, but it is a technical choice only and we do not claim to have estimates on the corresponding horizontal line segment, not being part of $\Gamma$.)

Let $i\alpha \in \A $ be given arbitrarily. The translated zeta-zero $\rho+i\alpha$ has imaginary part $\gamma+\alpha$, and so the set of vertical coordinate values $t$ for which any point of the horizontal line segment $[a+it,2+it]$ get closer to this zero than $d$ is restricted to the coordinates in $[\gamma+\alpha-d,\gamma+\alpha+d]$, having measure $2d$. So if we can specify some $t_1\in [4+d,5-d]$ such that it avoids this $2d$-measure prohibited interval for each $\alpha\in\A$ and $\rho$ with $\beta\ge p$, then $t_1$ is appropriate for our first choice. Now $[\gamma+\alpha-d,\gamma+\alpha+d]$ meshes into $[4+d,5-d]$ only in case $\gamma \in [4-\alpha,5-\alpha]$, so that the total measure of prohibited $t$-values for $t_1$ is at most $\sum_{i\alpha\in\A} 2d N(p,4-\alpha,5-\alpha) \le 2d n \max_{R \in [4-B,4+B]} N(p,R,R+1)= 2dn \max_{0\le R\le 4+B} N(p,R,R+1)$. If $B\le 1$, then these zero numbers are below $N(p,6)$, estimated by \eqref{Np6}, and if $B\ge 1$ then $B+4\ge 5$ and Lemma \ref{c:zerosinrange} applies, resulting in the estimate \eqref{Np5R} with $R:=B+4$. All these are below $X_0(B+4)$ (even if $B\le 1$ and $B+4\in [4,5)$, in which case we need only $X_0(B+4)\ge X$), so that the total measure of prohibited $t$-values for $t_1$ is at most $2dnX_0(B+4)$, while the full measure of $[4+d,5-d]$ is $1-2d$. Now excluding also all possible intervals of maximum length $2d$ containing values $|t+\al|<d$ for any $\alpha\in \AA$, we will still have a measure at least $1-2d-n2d>1/2$. Thus there is room for an admissible $t_1$ whenever $2nd X_0(B+4)$ stays below $1/2$, or, equivalently, if
$$
d < \frac{1}{4n X_0(B+4)}=\frac{b-\theta}{4n \left\{\log(B+5)+ 25 \log(A+\kappa)+12.5 \log\dfrac{1}{b-\theta} +57\right\}}.
$$
This is clearly guaranteed by \eqref{ddist-corr} with $t=t_1\in [4,5]$ because $\log(B+5)\le \log(|t|+B+5)$.

\bigskip
The argument for the choice of any $t_k$ ($k\ge 2$) is analogous. The only difference here is that in place of $[4+d,5-d]$ we now look for an admissible $t_k$ in the segment $[t_{k-1}+1+d,t_{k-1}+2-d]$, again of measure $1-2d$, so that after exclusion of all points with $|t+\al|<d$, still of measure $>1/2$. Noting that the value of $R:=R_k:=t_{k-1}+1$ is at least 5, the measure of the prohibited set can now be estimated by $2dn X_0(t_{k-1}+B+1)$, resulting in the condition
$$
d < \frac{1}{4n X_0(t_k+B)},
$$
in view of $t_{k-1}+1\le t_k$. This is again clearly guaranteed by \eqref{ddist-corr} with $t=t_k$.

\bigskip
The room for the choice of some appropriate $\si_k$ for the vertical line segments $V_k:=[\si_k+it_k,\si_k+it_{k+1}]$ is a little tighter. Again, it will suffice to ascertain the existence of an appropriate $\si_k$ for each $k\ge 0$ in view of the symmetry of the configuration.

The starting step is to pick $\si_0$ from the interval $[a,b]$ having measure $b-a=(b-\theta)/2$, such that $V_0:=[\si_0,\si_0+it_1]$ -- and whence by symmetry also the reflected segment and their union $[\si_0-it_1,\si_0+it_1]$ -- avoids all translated zeroes $\rho+i\alpha$ by at least $d$, and, further, it also avoids the translated poles $1+i\al$ by at least $d$. To ensure the latter requirement, we will simply take $\si \in [a,b-d]$, so that the distance from translated poles will be at least $d$, always. So given that $t_1\le 5-d$, it thus suffices to ensure that $W_0:=[\si_0,\si_0+i(5-d)]$ avoids all translated zeroes $\rho+i\alpha$ by at least $d$.

As above, consider any translated zero $\rho+i\alpha$ with $i\alpha\in \A$. Then in case $\Im (\rho+i\alpha)=\gamma +\alpha \not \in [-d,5]$, we necessarily see an avoidance by at least $d$: so that we need to take into account only zeta-zeroes in the rectangle $[a-d,b+d]\times [(-\alpha-d)i,(-\alpha+5)i] \subset [p,1]\times [(-\alpha-d)i,(-\alpha+5)i]$. Therefore, the number of possibly interfering zeroes does not exceed $\max_{-B \le \alpha \le B} N(p,-\alpha-d,-\alpha+5)=\max_{0 \le \alpha \le B} N(p,\alpha-d,\alpha+5)$.

The zeroes accounted for in $N(p,\alpha-d,\alpha+5)$ are covered by the ones accounted for in $N(p,5)$ and in $N(p,\max(\alpha-d,5),\alpha+5)$ together. So with any given $\alpha\ge 0$ let us put $R:=\max(\alpha-d,5)\ge 5$ and $T:=\alpha+5$; then $T-R\le 5+d\le 5.01$ and \eqref{zerosbetween} of Lemma \ref{c:zerosinrange} furnishes
\begin{align}\label{Np10R}
N(p,R,T)
& \le \frac{1}{0.49(b-\theta)} \left\{ \frac{20.04}{3\pi} \left(\log T + \log\left(\frac{11.4 (A+\kappa)^2}{0.49(b-\theta)}\right) \right) + \frac{16}{3}  \log\left(\frac{60 (A+\kappa)^2}{0.49(b-\theta)}\right)\right\} \notag
\\ & \le
\frac{1}{b-\theta} \left\{4.4\log (B+5) + 30.5 \log(A+\kappa)+25\log\dfrac{1}{b-\theta}+66\right\}
\end{align}
Estimating $N(p,5)$ by \eqref{zerosinth-corr} in Lemma \ref{l:Littlewood} -- similarly to \eqref{Np6} -- and adding we are led to
\begin{align}\label{Np10-full}
N(p, &\alpha-d,\alpha+5) \le \frac{1}{p-\theta}
\left\{\frac{1}{2} \log 5 +2 \log(A+\kappa) + \log\frac{1}{p-\theta} + 3 \right\} 5 +N(p,R,T)
\notag \\& <
\frac{1}{b-\theta} \left\{4.4 \log (B+5) + 51\log(A+\kappa)+31\log\dfrac{1}{b-\theta}+112.5\right\}=:Y_0(B)
\end{align}
Therefore the arc measure of $\si$-values closer than $d$ to the real part $\beta$ of any of the translated zeroes $\rho+i\alpha$ with $\rho\in [p,1]\times [i(\alpha-d),i(\alpha+5)]$ and $\alpha \in \A$ does not exceed $2dnY_0(B)$, while the total measure of available values for $\si_0$ is $(b-d)-a=\frac{b-\theta}{2}-d$, recalling that we have already excluded also the interval $[b-d,b]$, possibly too close to some translated pole $1+i\al$. Note that with this last exclusion we guarantee that no translated pole can get closer than $d$ to any point $\si_0+it$, whatever is $t\in \RR$.
So finally it suffices to have
$$
d \le \frac{b-\theta}{4nY_0(B)+2}< \frac{(b-\theta)^2}{4n \left(4.4 \log(|t|+B+5) + 51 \log (A+\kappa) + 31 \log\frac{1}{b-\theta}+ 113\right)},
$$
guaranteed in the condition \eqref{ddist-corr} of the statement.

\bigskip
Finally, for the selection of the coordinates $\si_k$ with $k\ge 1$ we work similarly. We are to choose a coordinate $\si_k\in[a,b]$ so that the vertical segment $W_k:=[\si_k+it_k,\si_k+i(t_k+2-d)]$ -- and then consequently also $V_k\subset W_k$ -- avoids all translated zeroes $\rho+i\alpha$ of $\ze$ by at least $d$. Recall that the value $t_k$ was chosen\footnote{At least for $k\ge 2$. For $k=1$, technically there is a problem with the argument because of $t_0:=0$, and not $3$. But we can conduct this calculation taking, technically, a modified $t_0:=3$ here, so that the choice $t_1$ is in the right interval. We leave to the reader the checking of the necessary technical modifications and that the above obtained estimates remain valid also for $k=1$.} from $[t_{k-1}+1+d,t_{k-1}+2-d]$, so that any zeta-zero closer than $d$ must have imaginary part $\gamma+\alpha \in  [t_{k-1}+1,t_{k-1}+2]$ and real part $\beta\ge a-d > p$.

The number of all these zeroes is $N(p,t_{k-1}+1-\alpha,t_{k-1}+2-\alpha)$. As $\alpha\in[-B,B]$
we find $N(p,t_{k-1}+1-\alpha,t_{k-1}+2-\alpha) \le \max_{t_{k-1}-B +1 \le R\le t_{k-1}+B+1} N(p,R,R+1)= \max_{0 \le R\le t_{k-1}+B+1} N(p,R,R+1) \le \max(N(p,6) ~,~\max_{5 \le R \le t_{k-1}+B+1} N(p,R,R+1)) \le X_0(\max(6,t_{k-1}+B+1))$, according to \eqref{Np6} and \eqref{Np5R} above. Therefore noting that $t_{k-1}+1\le t_k$ and using the monotonicity of $X_0$ again, we find that the prohibited set is of measure $\le 2nd X_0(\max(6,t_k+B))\le X_0(t_k+B+2)$ -- using $t_k\ge 4$, i.e. $t_k+2\ge 6$ here -- while the room available for our choice of $\si_k$ is again $b-d-a=\frac{b-\theta}{2}-d$. Whence there is a choice for $\si_k$ whenever
$d  < \frac{b-\theta}{4n X_0(t_k+B+2)+2}$.

Here we have for $s\in V_k \subset \Gamma$ that $t_k \le t \le t_{k+1}\le t_k+2$, so that the term $\log(R+1)=\log((t_k+B+2)+1)$, occurring in $X_0(t_k+B+2)$, can be estimated from above by $\log(t+B+3)$. Thus for finding an appropriate value of $\si_k$ it also suffices to have
$$
d < \frac{(b-\theta)^2}{4n \left\{\log (t+B+3) + 25 \log(A+\kappa)+12.5 \log\dfrac{1}{b-\theta}+57  \right\}+2(b-\theta)}.
$$
Here it becomes clear that this is always satisfied whenever \eqref{ddist-corr} holds, therefore we indeed have an appropriate choice for $\si_k \in[a,b]$, as wanted.
\end{proof}

\bigskip
\begin{lemma}\label{l:zzpongamma-c} For any $0<\theta<b<1$ and symmetric to $\RR$ translation set $\A\subset [-iB,iB]$, on the broken line $\Gamma=\Gamma_b^{\A}$, constructed in the above Lemma \ref{l:path-translates}, as well as on the horizontal line segments $H_k:=[a+it_k,2+it_k]$ and  $\overline{H_k}$, $k=1,\dots,\infty$ with $a:=\frac{b+\theta}{2}$, we have uniformly for all $\alpha \in \A$
\begin{equation}\label{linezest-c}
\left| \frac{\ze'}{\ze}(s+i\alpha) \right| \le n
\frac{1-\theta}{(b-\theta)^{3}} \left(6 \log(|t|+B+5)+60\log(A+\kappa) + 40 \log\frac1{b-\theta}+ 140\right)^2.
\end{equation}
\end{lemma}
\begin{proof} By symmetry, it suffices to work out the estimates on the upper halfplane i.e. for an arbitrary $s\in \Gamma_{+} \cup (\cup_{k=1}^\infty H_k)$. Let us fix such a value of $s=\si+it$, as well as an $\alpha\in \A$.

Formula \eqref{zlogprime} and \eqref{zlogprime-tsmall} of Lemma \ref{l:borcar} provide us with $z:=s+i\al=\si+i(t+\alpha)$ and $\epsilon=1$ or $0$ according to $|t+\al|<5.3$ or not, respectively, the estimate
\begin{align}\label{logzetaprimesum} \notag
\left| \frac{\ze'}{\ze}(s+i\alpha) \right| & \le
\left|\sum_{\rho\in S} \frac{1}{(s+i\al)-\rho}\right|  + \epsilon \left|\frac{1}{(s+i\al)-1}\right| +
\\ & + \frac{9(1-\theta)}{(\si-\theta)^2}
\left(34+14\log(A+\kappa)+18\log \frac{1}{\si-\theta} + 5\log_{+} |t+\al|\right),
\end{align}
where $S$ stands for the multiset of $\ze$-zeroes within distance $\de:=\frac{\si-\theta}{3}$ from $z=s+i\al$.

Note that the leftmost point of this neighborhood of $z=s+i\al$ is $\frac{2\si+\theta}{3}+i(t+\al)$, still lying to the right of $q+i(t+\al)$ with $q:=\frac{2a+\theta}{3}=\frac{b+2\theta}{3}$. Therefore these neighboring zeros belong also to the disk centered at $p+i(t+\al)$ (where $p:=1+(1-\theta)$) and of radius $r:=p-q$, with the above $q$. Referring to Remark \ref{r:alsoindisk} provides that the number $N$ of these zeroes is thus estimated by
\begin{align*}
N:=N(S) & \le
\frac{1-\theta}{\frac32(q-\theta)}\left(\log_{+}|t+\al| + 6\log(A+\kappa) + 6\log\frac1{1-\theta}+14\right)
\notag \\ & \le \frac{1-\theta}{b-\theta}\left(2\log_{+}(t+B) + 12\log(A+\kappa) + 12\log\frac1{1-\theta}+28\right).
\end{align*}
Further, the absolute value of each term in the sum $\sum_{\rho\in S} \frac{1}{(s+i\al)-\rho}$ over $S$ is at most $1/d(t)$ with the value of $d(t)$ given in \eqref{ddist-corr}, whence\footnote{Using also the inequality $(3X+30Y+20Z+59)^2\ge (2X+12Y+12Z+28)(4.4+51Y+31Z+113)$, for all $X,Y,Z\ge 0$.} we get
\begin{align*}
\left|\sum_{\rho\in S} \frac{1}{(s+i\al)-\rho}\right|
\le &\frac{4n(1-\theta)}{(b-\theta)^3}  \left(2 \log(t+B+5)+12\log(A+\kappa) +12\log\frac1{b-\theta}+28\right)
\\ & \qquad \times\left(4.4 \log(t+B+5) + 51 \log (A+\kappa) + 31 \log\frac{1}{b-\theta}+ 113\right)
\\ & \le \frac{4n(1-\theta)}{(b-\theta)^3} \left(3\log(t+B+5)+30\log(A+\kappa) + 20 \log\frac1{b-\theta}+ 59\right)^2.
\end{align*}
The rightmost expression in \eqref{logzetaprimesum} is an order of magnitude smaller. Indeed, writing in $\si\ge a =\frac{b+\theta}{2}$, $\log\frac{1}{\si-\theta}\le \log\frac{2}{b-\theta}$, $\log_{+}(|t+\alpha|)\le \log(t+B+5)$ and $1\le n/(b-\theta)$ we get the upper bound
\begin{align*}
&\le
\frac{4n(1-\theta)}{(b-\theta)^3} \cdot 9 \cdot \left(5\log_{+}(t+B)+14\log(A+\kappa)+18\log \frac{1}{b-\theta} + 47\right)
\\
& <
\frac{4n(1-\theta)}{(b-\theta)^3} \cdot 18 \cdot \left(3\log(t+B+5)+30\log(A+\kappa) + 20 \log\frac1{b-\theta}+ 59\right).
\end{align*}

Finally, in case $\epsilon=1$, we have a term $|1/(\si+i(t+\alpha)-1)|\le 1/d(t)$, with the value of $d(t)$ given in \eqref{ddist-corr}.
Let us put $X:=3\log(t+B+5)+30\log(A+\kappa) + 20 \log\frac1{b-\theta}+ 59$. Then it is clear that we have
$\frac{1}{d} \le \frac{4n(1-\theta)}{(b-\theta)^3}  \left(4.4 \log(|t|+B+5) + 51 \log (A+\kappa) + 31 \log\frac{1}{b-\theta}+ 113\right) \le \frac{4n(1-\theta)}{(b-\theta)^3} 2X$.
Collecting our estimates and substituting into \eqref{logzetaprimesum} we thus get
$$
\left| \frac{\ze'}{\ze}(s+i\alpha) \right| \le \frac{4n(1-\theta)}{(b-\theta)^3} \left\{ X^2 + 18X + 2X \right\}
= \frac{4n(1-\theta)}{(b-\theta)^3} (X^2+20 X) \le \frac{4n(1-\theta)}{(b-\theta)^3} (X+10)^2.
$$
The assertion of the Lemma is proved.
\end{proof}

\section{A Riemann-von Mangoldt type formula of prime distribution}\label{sec:sumrho}

In the following we will denote the set of $\zeta$-zeroes, lying
to the right of $\Gamma$, by $\Z(\Gamma)$, and denote
$\Z(\Gamma,T)$ the set of those zeroes $\rho=\beta+i\gamma\in
\Z(\Gamma)$ which satisfy $|\gamma|\leq T$.

\begin{theorem}[Riemann-von Mangoldt formula]\label{l:vonMangoldt}
Let $\theta<b<1$ and $\Gamma=\Gamma_b^{\{0\}}$ be the curve defined in Lemma \ref{l:path-translates} for the one-element set $\A:=\{0\}$ with $t_k$ denoting the corresponding set of abscissae in the construction.
Then for any $k=1,2,\ldots$, and $4 \leq t_k<x$ we have
$$
\psi(x)=x - \sum_{\rho \in \Z(\Gamma,t_k)}
\frac{x^{\rho}}{\rho} + O\left( \frac{1-\theta}{(b-\theta)^{3}}
\left(A+\kappa+\log \frac{x}{b-\theta}\right)^3 x^b \right).
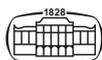
$$
\end{theorem}

\begin{proof} Analogously to the effective version of the classical
Perron formula, as given in e.g. \cite{T}, Th\'eor\`eme 2,
Chapitre II.2, p. 135, for any fixed $T,x>1$, $1<p<2$ one can
prove the effective representation
$$
\psi(x)=
\frac{1}{2\pi i} \int_{p-iT}^{p+iT} -\frac{\zeta'}{\zeta}(s)
\frac{x^s}{s}ds + O\left( x^p \sum_{g\in\G}^\infty
\frac{\Lambda(g)}{|g|^p(1+T |\log(x/|g|)|)}\right).
$$
Here we move the path of integration to the curve $\Gamma$ of Lemma
\ref{l:path-translates}: more precisely, to the curve $\Gamma_T$ consisting of
the part of $\Gamma$ in the strip $-T\leq \Im s \leq T$ joined by
the two segments $[p-iT,q-iT]$ and $[q+iT,p+iT]$, with $q\pm iT\in
\Gamma$. Obviously, the set of zeroes of $\zeta$ in the domain encircled
by $[p-iT,p+iT]$ and $\Gamma_T$ is $\Z(\Gamma,T)$.
Assuming that no zero lies on $\Gamma_T$ it follows by an application
of the residue theorem
\begin{equation}\label{psifirstform}
\psi(x)=x-\sum_{\rho \in \Z(\Gamma,T)}
\frac{x^{\rho}}{\rho}+\int_{\Gamma_T} -\frac{\zeta'}{\zeta}(s)
\frac{x^s}{s}ds + O\left(\int_1^{\infty}
\frac{x^p}{u^p(1+T|\log\frac{x}{u}|)} d\psi(u)\right).
\end{equation}
Now let us put $T=t_k (\geq 4)$ and write $a:=\frac{b+\theta}{2}\in (\theta,1)$. Since the arc length of
$\Gamma_{t_k}$ is at most $3t_k$, by an application of Lemma
\ref{l:zzpongamma-c} and using that $|1/s|\le 1/\sqrt{a^2+t^2}$ for all $s=\si+it\in\Gamma$ we find
\begin{align}\label{intGammaT}
\int_{\Gamma_T} \left|\frac{\zeta'}{\zeta}(s)\frac{x^s}{s}\right||ds|
& \ll \frac{1-\theta}{(b-\theta)^{3}} \left(\log(A+\kappa)+1+\log t_k +
\log \frac1{b-\theta}\right)^2 \left(\frac{x^p}{t_k} +
x^b \int_0^{t_k} \frac1{\sqrt{a^2+u^2}} du\right) \notag \\
&\ll \frac{1-\theta}{(b-\theta)^{3}} \left(\log(A+\kappa)+1+\log t_k +
\log \frac{1}{b-\theta}\right)^2 \left(\frac{x^p}{t_k} +
x^b \log (t_k/a) \right),
\end{align}
Note that $a=\frac{b+\theta}{2}>b-\theta$, so that $\log(t_k/a)\le \log t_k +\log\frac{1}{b-\theta}$, and that $\log(A+\kappa)+1+\log\frac{1}{b-\theta}+\log(t_k+5) \ll (A+\kappa)+\log(\frac{x}{b-\theta})$ for $t_k\ge 4$, whence we get finally
\begin{equation}\label{intGammaTfin}
\int_{\Gamma_T} \left|\frac{\zeta'}{\zeta}(s)\frac{x^s}{s}\right||ds| \ll \left((A+\kappa)+\log(\frac{x}{b-\theta}) \right)^3 \left(\frac{x^p}{t_k} +
x^b \log (t_k) \right).
\end{equation}
For the $O$-term in \eqref{psifirstform} we execute a detailed calculus here. First,
observe that by the definition of $\Lambda(g)$ we have
$d\psi(u)\leq \log u ~d{\N}(u)$, (the inequality interpreted as
between positive measures), whence it suffices to treat
\begin{equation}\label{Jdef}
J:=J(x,p,T):=x^p \int_1^{\infty}
\frac{\log u}{u^p(1+T|\log(x/u)|)} d\N(u).
\end{equation}
First we approximate $\log u$ in the integral by $\log x$,
with an error
\begin{equation}\label{logxuerror} \notag
\int_1^{\infty} \frac{|\log (x/u)|}{u^p(1+T|\log(x/u)|)} d\N(u)
\leq \int_1^{\infty} \frac{1}{Tu^p} d\N(u)=\frac{1}{T}\zeta(p)\leq
\frac{A+\kappa}{T}\frac{p}{p-1} ,
\end{equation}
invoking \eqref{zsintheright} in the last step. Writing in
$\N(u)=\kappa u + \R (u)$, we thus find
\begin{equation}\label{Japprox}
\left| J -\kappa \log x I - x^p \log x L\right| \leq
\frac{2(A+\kappa)}{(p-1)T}x^p,
\end{equation}
where
\begin{equation}\label{Idef}
I:= \int_1^\infty \left(\frac{x}{u}\right)^p \frac{du}{1+T|\log(x/u)|}
\end{equation}
and
\begin{equation}\label{Ldef}
L:= \int_1^\infty\frac1{u^p}
\frac{d\R(u)}{1+T|\log(x/u)|}.
\end{equation}
Cutting the interval of integration at $x$, we denote $I'$ and
$L'$ the integrals on $[1,x]$, and $I"$, $L"$ the integrals on
$[x,\infty]$. On the interval $[1,x]$ the absolute value is $\log(x/u)$, and conversely, on $[x,\infty]$ it is $\log(u/x)$.
On both parts, the integrand is a nice, continuously differentiable
function. Substituting $v:=x/u$ and then $w:=T\log v$ we obtain
$$
I'=\int_1^x \frac{xv^{p-2}dv}{1+T\log v}=\frac{x}{T} \int_0^{T\log
x}\frac{e^{\al w}dw}{1+w}\qquad\textrm{with}\quad \al:=\frac{p-1}{T}
$$
Estimating the exponential by $e$ up to $1/\al-1$, the part until
$1/\al-1$ is at most $e\log(1/\al)$. The rest is at most
$$
\int_{1/\al-1}^{T\log x} \al e^{\al w}dw < e^{\al T\log x}=x^{\al
T},
$$
whence
\begin{equation}\label{Iprime}
I'\leq  \frac{e x\left(\log T+\log\frac{1}{p-1}\right)}{T} + \frac{x^p}{T}
\end{equation}
To estimate $I"$ we substitute first $v:=u/x$, and second
$w:=T\log v$ to find
$$
I" =x \int_1^{\infty} \frac{dv}{v^p(1+T\log v)} = \frac{x}{T}
\int_0^{\infty} \frac{e^{-\al w} dw}{1+w}.
$$
with the very same $\al$. Here we can calculate further as
\begin{align}\label{alphacalc}
\int_0^{\infty} \frac{e^{-\al w} dw}{1+w}
& =\left[
\frac{-e^{-\al w}}{\al (1+w)} \right]_0^{\infty} -\frac{1}{\al}
\int_0^{\infty} \frac{e^{-\al w} dw}{(1+w)^2} = \frac{1}{\al}
\int_0^{\infty} \frac{1-e^{-\al w} }{(1+w)^2}dw \notag \\&
\leq \frac{1}{\al} \int_0^{\infty} \frac{\min(1,\al w)}{(1+w)^2}dw =
\int_0^{1/\al} \frac{w ~dw}{(1+w)^2} + \frac{1}{\al}
\int_{1/\al}^{\infty} \frac{dw}{(1+w)^2} \notag \\& \leq
\log(1+1/\al) +\frac{1}{1+\al} \leq \log(1/\al) + 2
\\ &= \log T +\log \frac{1}{p-1} + 2. \notag
\end{align}
So in all we obtain
\begin{equation}\label{Iest}
I\ll \frac{x^p + x\log T+x\log\frac{1}{p-1}}{T}.
\end{equation}
In case of $L$ we must pay some attention to the possible sign
changes of $\R$, although the order of this term is in general
smaller. Now
\begin{align*}
L'& =\left[ \frac{\R(u)}{u^p(1+T\log(x/u))}\right]_1^x + \int_1^x
\frac{\R(u)\left({p-T(1+T\log(x/u))^{-1}}\right)}{u^{p+1}(1+T\log(x/u))}
du
\end{align*}
so
\begin{equation}\label{Lprime1}
\left| L'-\frac{\R(x)}{x^p} \right| \leq
\int_1^x\frac{|\R(u)|\left({p(1+T\log(x/u))+T}\right)}
{u^{p+1}(1+T\log(x/u))^2} du .
\end{equation}
Writing in $|\R(u)|\leq Au^\theta$ and substituting $v:=x/u$ we
obtain\footnote{To derive the last line from the last but one, we
calculate the integral as follows. $\int_T^{T\log x}
\frac{e^{\beta\xi}}{\xi}d\xi = \left[ \frac{e^{\beta\xi}}{\beta
\xi} \right]_T^{T\log x} + \int_T^{T\log x}
\frac{e^{\beta\xi}}{\beta\xi^2}d\xi < \frac{e^{(p-\theta)\log
x}}{(p-\theta)\log x} + \int_{p-\theta}^{(p-\theta)\log x}
\frac{e^vdv}{v^2}$. If $p-\theta <1$, the part of the last
integral below $v=1$ does not exceed $e/(p-\theta)<e/(p-1)$.
Here we distinguish two cases, according to $(p-\theta)\log x\geq
1$ or not. In the first case the integral for $1\leq v \leq
(p-\theta)\log x$ is at most $e^{(p-\theta)\log
x}\int_1^{(p-\theta)\log x} \frac{dv}{v^2}<x^{p-\theta}$, and also
$\frac{e^{(p-\theta)\log x}}{(p-\theta)\log x}\leq x^{p-\theta}$,
which yields the estimate $\ll 1/(p-1)+x^{p-\theta}$. On the other
hand for $(p-\theta)\log x<1$ we have $\frac{e^{(p-\theta)\log
x}}{(p-\theta)\log x} < \frac{e}{(p-1)\log 4} \ll 1/(p-1)$ and the
overall estimate is of the same order.}
\begin{align}\label{Lprime}
\left| L'-\frac{\R(x)}{x^p} \right| & \leq
x^{\theta-p-1} \int_1^x\frac{Av^{1+p-\theta}\left({2(1+T\log v)+T}\right)}
{(1+T\log v)^2} \frac{x dv}{v^2} \notag  \\
& \leq 2AT x^{\theta-p} \int_1^x\frac{v^{p-\theta}(1+\log v)}
{(1+T\log v)^2} \frac{dv}{v} \notag \\& = 2AT x^{\theta-p} \int_0^{\log x}
e^{(p-\theta)w} \frac{1+w}{(1+Tw)^2} dw \notag \\
& = 2 AT x^{\theta-p} \left( \int_0^1 +\int_1^{\log x} \right)
\\ & \ll
AT x^{\theta-p}\int_0^1 \frac{dw}{(1+Tw)^2} + A x^{\theta-p} \int_1^{\log x} \frac{e^{(p-\theta)w}}{1+Tw}
dw \notag \\
& = Ax^{\theta-p} \left(1 -\frac{1}{T+1}+\frac1T \int_T^{T\log x}
\frac{e^{\beta \xi}}{1+\xi} d\xi\right) \qquad \left(\textrm{with}
~\beta:=\frac{p-\theta}{T}\right) \notag \\ & \ll Ax^{\theta-p}
\left(1+\frac{1}{(p-1)T} + \frac{x^{p-\theta}}{T}\right) =
Ax^{\theta-p}\left(1+\frac{1}{(p-1)T} \right) + \frac{A}{T}
\notag.
\end{align}
Similarly, integration by parts gives
\begin{align*}
L"& =\left[ \frac{\R(u)}{u^p(1+T\log(u/x))}\right]_x^{\infty} +
\int_x^{\infty}
\frac{\R(u)\left({p+T(1+T\log(u/x))^{-1}}\right)}{u^{p+1}(1+T\log(u/x))}
du
\end{align*}
whence with the very parameter $\beta$
\begin{align}\label{L2prime}
\left| L" + \frac{\R(x)}{x^p} \right| & \leq A \int_x^{\infty}
\frac{u^{\theta-p-1}\left|{T+p(1+T\log(u/x))}\right|}
{(1+T\log(u/x))^2} du \notag \\ &= A x^{\theta-p} \int_1^{\infty}
\frac{v^{\theta-p-1}\left|{T+p(1+T\log v)}\right|} {(1+T\log v)^2}
dv \notag \\ &\leq A x^{\theta-p}  \int_0^{\infty} \frac{e^{-\beta
w}p \left({T+1+w}\right)} {(1+w)^2} \frac{dw}{T} \notag \\ & \leq
2 A x^{\theta-p} \int_0^{\infty} {e^{-\beta
w}}\left(\frac1{T(1+w)}
+\frac{1}{(1+w)^2}\right) {dw} \notag \\
& \leq 2 A x^{\theta-p} \left( \frac1T \log\frac{T}{p-\theta} +
\frac 2T +1 \right)
\end{align}
estimating the integral as in \eqref{alphacalc}.

Adding \eqref{Lprime} and \eqref{L2prime} we thus arrive at
\begin{equation}\label{Lestimate}
|L| \ll \frac{A}{T} + A x^{\theta-p}\left(1+\frac{1}{(p-1)T} +
\frac{\log\frac{T}{p-1}}{T}\right).
\end{equation}
Taking into account \eqref{Iest} and \eqref{Lestimate} in
\eqref{Japprox}, we are led to
\begin{align}\label{Jfinal}
J & \ll \frac{(A+\kappa)x^p}{(p-1)T} + \kappa \log x \frac{x^p +
x\log \frac{T}{p-1}}{T} + A x^p \log x \left( \frac{1}{T} +
x^{\theta-p}+ \frac{x^{\theta-p}}{(p-1)T} + \frac{
x^{\theta-p}\log\frac{T}{p-1}}{T}\right)\notag \\ & \ll (A+\kappa)
\left( \frac{x^p \log x}{T} + \frac{x^p+x^\theta\log x}{(p-1)T}+ x \log x
\frac{\log\frac{T}{p-1}}{T} + x^\theta\log x \right).
\end{align}
Let us take now $p:=1+1/\log x$ and apply the above for $T=t_k\ge 4$ and $x\ge t_k$ to get
\begin{equation}\label{Jfinalsimpler}
J \ll \left((A+\kappa)+\log x + \log\frac{1}{b-\theta}\right)^3 \left(\frac{x}{t_k}+x^\theta\right).
\end{equation}

Finally, using the estimates of \eqref{intGammaTfin} (applied for $T:=t_k$)
and \eqref{Jfinalsimpler} in \eqref{psifirstform}, and taking into account that $4\le t_k\le x$ implies $x/t_k\le 4x^b$, the
assertion of the Theorem follows.
\end{proof}

\section*{Acknowledgements}
Supported in part by Hungarian National Research, Technology and Innovation Office Grant \# T-72731,
T-049301, K-61908, K-81658, K-100461, K-104183, K-109789, K-119528, and K-132097 by the Hungarian-French Scientific and Technological Governmental Cooperation, Project \# T\'ET-F-10/04 and the Hungarian-German Scientific and Technological Governmental Cooperation, Project \# TEMPUS-DAAD \# 308015.


\end{document}